\documentclass[10pt]{article}%
\usepackage{amsmath}
\usepackage{amssymb}
\usepackage{amsbsy}
\usepackage{amsthm}
\usepackage{epsfig}
\usepackage{wrapfig,booktabs}
\usepackage{xcolor,float}

\numberwithin{equation}{section}
\newcommand{\norm}[1]{\left\Vert#1\right\Vert}
\newtheorem{theorem}{Theorem}[section]

\newtheorem{lemma}{Lemma}[section]

\newtheorem{algorithm}{Algorithm}[section]

\newcommand{\s}{\mathcal{S}}

\newcommand{\abs}[1]{\left\vert#1\right\vert}

\newcommand{\bu}{{ u}}
\newcommand{\bB}{{ B}}
\newcommand{\bv}{{ v}}
\newcommand{\bff}{{ f}}
\newcommand{\bfg}{{ g}}
\newcommand{\bw}{{ w}}

\newcommand{\bfV}{{\bf V}}

\newcommand{\bX}{{X }}

\usepackage{accents}

\begin{document}

\date{}

\author{Aytekin Cibik \footnotemark[1]\, \hspace{.1in}Fatma G. Eroglu \footnotemark[2] \hspace{.1in} Song\"{u}l Kaya\footnotemark[3]  }


\title{Numerical analysis of an efficient second order time filtered backward Euler method for MHD equations}

\maketitle

\renewcommand{\thefootnote}{\fnsymbol{footnote}}

\footnotetext[1]{Department of  Mathematics, Gazi University, 06550 Ankara, Turkey; email: abayram@gazi.edu.tr}
\footnotetext[2]{Department of Mathematics, Faculty of Science, Bart{\i}n University, 74110 Bart{\i}n, Turkey; email: fguler@bartin.edu.tr.}

\footnotetext[3]{Department of Mathematics, Institute of Applied Mathematics, Middle East
Technical University, 06800, Ankara, Turkey; email:
smerdan@metu.edu.tr}


\begin{abstract}
The present work is devoted to introduce the backward Euler based modular time filter method for MHD flow. The proposed method improves the accuracy of the solution without a significant change in the complexity of the system. Since time filters for fluid variables are added as  separate post processing steps, the method can be easily incorporated into an existing backward Euler code. We investigate the conservation and long time stability properties of the improved scheme. Stability and second order convergence of the method are also proven. The influences of introduced time filter method on several numerical experiments are given, which both verify the theoretical findings and illustrate its usefulness on practical problems.
\end{abstract}


\section{Introduction}
This paper considers a modular time filter method combined with the backward Euler method for the magnetohydrodynamics (MHD) flow problems. A simple method of incorporating this time filter into an existing code is to add extra lines for each fluid variables, thus it can be considered as a post processing step. As discussed in \cite{GL18}, for ODEs adding such time filter to backward Euler not only increases accuracy from first order to second order, but also reduces spurious oscillations of numerical solutions, preserves A-stability of the method and yields a useful error estimator.
Recently, the time filter of \cite{GL18} was considered for Navier-Stokes equations for constant and variable time steps by DeCaria, Layton and Zhao in \cite{DLZ19}, resulting a stable, second order time accurate adaptive method with a low complexity.

The goal of this paper is to extend this novel idea from  \cite{DLZ19} of time accurate flow approximation to the  MHD system for constant time steps,  which describes the mutual interaction between the magnetic field and electrically conductive fluids. These flows have diverse applications in, e.g., hydrology, geophysics, astrophysics and cooling system designs \cite{D01,HK09,B08,SA15}. It was first presented by Ladyzhenskaya and has been developed in \cite{CASE10,GLP04,GT05,GT07,LP10,LW0}. Using Navier Stokes equations (NSE) and Maxwell equations, the governing equations of MHD system are given by
\begin{eqnarray}
u_t  - Re^{-1}\Delta u + u\cdot\nabla u   - s B \cdot \nabla B + \nabla P & = & f, \label{mhd5}\\
\nabla \cdot u & = & 0,  \label{mhd6}\\
B_t - Re_m^{-1}\Delta B + u\cdot\nabla B - B\cdot\nabla u - \nabla \lambda & = &
\nabla \times g, \label{mhd7}\\
\nabla \cdot B & = & 0  \label{mhd8}
\end{eqnarray}
in a bounded polyhedral domain $\Omega \subset \mathbb{R}^d, d\in \{2,3\}$. Here, $u$, $P:=p+\frac{s}{2}|B|^2$, $p$  and $B$ denote the unknown velocity, modified pressure, pressure  and magnetic field, respectively. The body forces $\bff$ and $\nabla \times g$ are forcing on the velocity and magnetic field, respectively. Also, $Re$ is the Reynolds number, $Re_m$ is the magnetic Reynolds number, and $s$ is the coupling number. The Lagrange multiplier (dummy variable) $\lambda$ corresponds to the solenoidal constraint on the magnetic field.  In the continuous case, provided the initial condition $B_0$ is solenoidal, then the use of $\lambda$ is unnecessary, see \cite{CASE10}.  However, when discretizing with the finite element method, this solenoidal constraint  is needed to be enforced explicitly and thus the additional variable is required. We also assume that  the system \eqref{mhd5}-\eqref{mhd8} is equipped with homogeneous Dirichlet boundary conditions for velocity and the magnetic field.

Due to the coupling of the equations of the velocity and the magnetic field, developing efficient, accurate numerical methods for solving MHD system \eqref{mhd5}-\eqref{mhd8} remains a great challenge in computational fluid dynamics community. It is well known that time filter methods combined with leapfrog scheme are commonly used in geophysical fluid dynamics to reduce spurious oscillations to improve predictions, see e.g.\cite{R72}, but these methods degrade the numerical accuracy and over damps the physical mode. A successfully tuned  model was developed by Williams \cite{W09} reducing undesired numerical damping of \cite{R72} with higher order accuracy, see \cite{AK11,LT16,T14,W11} and references therein. On the other hand, in practice, the use of the backward Euler method is often preferred to extend a code for the steady state problem and this yields stable but inefficient time accurate transient solutions, see \cite{G98}. To improve this behavior,  time filters are used to stabilize the backward Euler discretizations in \cite{GL18} for the classical numerical ODE theory. 

The present work extends the method of \cite{DLZ19} tailored to MHD flows for constant time step. As it is mentioned in this study, the constant time step method is equivalent to a general second order, two step and A-stable method given in \cite{GR79} and \cite{JMR}. The scheme we consider is the time filtered backward Euler method,  which is  efficient,  $\mathcal{O}(\Delta t^2)$ and amenable to implementation in existing legacy codes. In addition, we also consider the numerical conservation of physically conserved quantities such as the energy and the helicity. It is worth noting that for ideal MHD with periodic boundary conditions, we prove both analytically and numerically the time filtered backward Euler method preserves the exact conservation of energy and helicity with the strong enforcement of the solenoidal constraints on the velocity and magnetic field. In addition, we prove the method's velocity and magnetic field are both stable and long time stable without any time step restriction.

This paper is arranged as follows. Section 2 gathers notations and preliminary results which will be used for the analysis. In Section 3, the time filtered backward Euler method is described along with the proof of conservation properties. Section 4 presents stability and convergence  analysis of for the fully discrete scheme. Numerical experiments are presented to verify theoretical results in Section 5. Finally, conclusions of the paper are given in Section 6.

\section{Notation and Preliminaries  }

Standard notations of Lebesgue and Sobolev spaces are used throughout this paper.  The inner product
of $(L^2(\Omega))^d$, will be denoted by $(\cdot,\cdot)$, the norm in $(L^2(\Omega))^d$ by $\|\cdot\|$ and the norm in the Hilbert space $(H^k(\Omega))^d$  by $\|\cdot\|_k$. For clarity of presentation, we assume no-slip boundary conditions.
We consider the classical function spaces
\begin{eqnarray}
X &=& (H_0^1(\Omega))^d :=\{v\in (L^2(\Omega))^d:\nabla v \in L^2(\Omega)^{d\times d}, \, v= 0\,\mbox{on}\,\partial \Omega\},
\nonumber
\\
Q &=& L_0^2(\Omega) := \{q \in L^2(\Omega): \int_{\Omega} q \ dx = 0\}. \nonumber
\end{eqnarray}
The norm of the dual space $H^{-1}$ of $X$ is denoted by $\|\cdot\|_{-1}$. As usual, one has $X\subset L^2(\Omega)\subset H^{-1}(\Omega)$ with compact injection.
The divergence free velocity space is given by
\begin{equation*}
  V:=\{v\in X, (\nabla \cdot v, q)= 0, \forall q\in Q\}.
\end{equation*}
We define the following norms for all Lebesgue measurable $w:[0,T]\to X$:
\begin{eqnarray}
	\norm{w}_{L^p(0,T;X)}&=&\bigg(\int_{0}^{T}\|w(t)\|_{X}^p dt\bigg)^{1/p},\quad 1\leq p<\infty\nonumber\\
	\norm{w}_{L^{\infty}(0,T;X)}&=&ess \sup_{0\leq t\leq T}\|w(t)\|_{X}.
\end{eqnarray}
In the error analysis, we use the Poincar\'e-Friedrichs' inequality,
\begin{equation}
\norm{v} \leq C_p \norm{\nabla v},\label{pof}
\end{equation}
for all $v\in X$, where $C_p$ is a constant depending only on the size of $\Omega$. The following properties of for the skew symmetric form are necessary in the  analysis.
\begin{lemma} \label{tribound} The trilinear skew-symmetric form $(u \cdot \nabla v, w)$ satisfies
\begin{eqnarray}
\begin{array}{rcl}
(u \cdot \nabla v, w)&\leq& C \norm{\nabla u} \norm{\nabla v}\norm{\nabla w}, \label{tribound1}\\
(u \cdot \nabla v, w)&\leq&C \norm{u}^{1/2} \norm{\nabla u}^{1/2} \norm{\nabla v} \norm{\nabla w}, \\
 (u \cdot \nabla v, v)&=&0\label{tribound2}
\end{array}
\end{eqnarray}
for all $ u,v, w\in X$.
\end{lemma}
\begin{proof}
 Utilizing  H\"older's inequality, interpolation theorem, Sobolev embedding theorem and Poincar\'e inequality gives the stated results, see \cite{WJL8}.
\end{proof}
We use conforming finite element spaces based on edge to edge triangulations of $\Omega$ (with maximium element diameter $h$) by $X_h\subset X$ and $Q_h\subset Q$. In the computations, we consider the Scott-Vogelius finite element spaces for velocity-pressure and magnetic field-Lagrange multiplier pairs. It is well known that on a barycenter refinement of regular mesh, this element satisfies the discrete inf-sup condition, see \cite{GR79} and the optimal approximation properties, \cite{Zhang05}. Since Scott-Vogelius elements enforce mass conservation pointwisely  for both velocity and magnetic field, e.g.
\begin{eqnarray} \label{divuB}
\nabla \cdot u_h^n & = & 0, \nonumber  \\
\nabla \cdot B_h^n & = & 0,
\end{eqnarray}
it has been successfully used for multiphysics problems, see e.g.\cite{BA08,CASE10}.

Following \cite{J16}, one admits the optimal approximation properties for the velocity and magnetic field.
\begin{eqnarray}
\inf_{v_{h} \in X_{h}} \left \{\| {u-v_{h}}\| + h \|
{\nabla (u-v_{h})}\| \right\} &\leq C h^{s+1} \norm {u} _{s+1}, \label{app1}
\\
\inf_{B_{h} \in X_{h}} \left \{\| {B-B_{h}}\| + h \|
{\nabla (B-B_{h})}\| \right\} &\leq C h^{s+1} \norm {B} _{s+1}.\label{app2}
\end{eqnarray}
The discretely divergence-free space is defined by
\begin{eqnarray*}
{V}_h=\{v_h\in X_h:(q_h,\,\nabla\cdot
v_h)=0, \,\forall\,q_h \in {Q}_h\},
\end{eqnarray*}
which is also the divergence-free subspace of $X_h$ when using Scott-Vogelius pair.

We also use the following space in the analysis:
\begin{eqnarray}
	L_{\infty}(\mathbb{R}_+,\bfV_h^*)=\{\bfg\in \Omega^d\times\mathbb{R}_+\to \mathbb{R}^d, \, a.e.\,\, \, t >0,\, \exists M< \infty, \, \|\bfg(t)\|_{\bfV_h^*}<M\},
\end{eqnarray}
where $\bfV_h^*$ and $\|\cdot\|_{\bfV_h^*}$ are dual spaces of $\bfV_h$ and its norm which is given by
\begin{eqnarray}
\|\bw\|_{\bfV_h^*}=\sup_{\bv_h\in \bfV_h^*}\dfrac{(\bw,\bv_h)}{\|\nabla\bv_h\|}.
\end{eqnarray}
 The following discrete Gronwall lemma, stated in \cite{HR4} plays an important role in the error analysis.
\begin{lemma}\label{gron}[Discrete Gronwall Lemma]
Let $\Delta t$, M, and $\alpha_n,\beta_n,\xi_n,\delta_n$ (for integers $n\geq 0$) be finite nonnegative numbers such that
\begin{eqnarray*}
\alpha_m +\Delta t \sum_{n=0}^{m} \beta_n \leq \Delta t \sum_{n=0}^{m} \delta_n\alpha_n +\Delta t \sum_{n=0}^{m} \xi_n + M\quad\mbox{for}\quad m\geq 0.
\end{eqnarray*}
Suppose $\Delta t \delta_n<1$ for all $n$, then
\begin{eqnarray*}
\alpha_m +\Delta t \sum_{n=0}^{m} \beta_n \leq \exp\Bigg(\Delta t \sum_{n=0}^{m}\beta_n \frac{\delta_n}{1-\Delta t \delta_n}\Bigg) \Bigg(\Delta t \sum_{n=0}^{m} \xi_n + M\Bigg)\quad\mbox{for}\quad m\geq 0.
\end{eqnarray*}
\end{lemma}

\section{Time Filtered MHD Equations }

Time filtered finite element scheme we study consists of two steps. In the first step the usual backward Euler method is applied to MHD equations and the second step includes the linear combination of solutions at previous time levels. As we will prove later, while the second step doesn't require additional function evaluations, it has a profound impact on the solution quality such that it increases time accuracy. By assuming the prescribed values are nodal interpolants of the fluid variables, we now present the finite element approximation of \eqref{mhd5}-\eqref{mhd8} for constant time step method. Let $T$ denote the final time, $M$ denote the number of time steps to take and define the time step ${\Delta T}=\frac{T}{M}$. The fully discrete solution at time $t_{n}=n\Delta t$, $n=0,1,2...M$ will be denoted by $u^n_h$ and $B^n_h$. The scheme applied to the problem \eqref{mhd5}-\eqref{mhd8} reads as follows:

\begin{algorithm}\label{algo}
Given $u_{h}^{n-1}, u_{h}^{n},B_{h}^{n-1},B_{h}^{n},P_{h}^{n-1},P_{h}^{n},\lambda_{h}^{n-1},\lambda_{h}^{n}$,
find $(u_{h}^{n+1},B_{h}^{n+1},P_{h}^{n+1},\\ \lambda_{h}^{n+1})
\in (X_h,X_h,Q_h,Q_h)$ satisfying\\
{\bf Step 1:}
\begin{eqnarray}
\frac{1}{\Delta t}(\tilde {\bu_h}^{n+1}-\bu_{h}^{n},v_{h}) +Re^{-1}(\nabla \tilde {\bu}_{h}^{n+1},\nabla \bv_{h})
+ (\tilde {\bu}_{h}^{n+1}\cdot \nabla \tilde {\bu}_{h}^{n+1},\bv_{h})
\nonumber\\
-s(\tilde{\bB}_{h}^{n+1}\cdot \nabla \tilde{\bB}_{h}^{n+1},\bv_{h})
-(\tilde{P}_{h}^{n+1},\nabla \cdot \bv_{h})
=(\bff(t^{n+1}),\bv_{h}),\label{BE1}\\
(\nabla\cdot \tilde {\bu}_{h}^{n+1},q_{h})= 0,\label{BE2} \\
\frac{1}{\Delta t}(\tilde{\bB}_{h}^{n+1}-\bB_{h}^{n},\chi_{h})
+Re_{m}^{-1}(\nabla \tilde{\bB}_{h}^{n+1},\nabla\chi_{h})
-(\tilde{\bB}_{h}^{n+1}\cdot \nabla  \tilde{\bu}_{h}^{n+1},\chi_{h})\nonumber\\
+(\tilde{\bu}_{h}^{n+1}\cdot \nabla \tilde{\bB}_{h}^{n+1},\chi_{h})
+(\tilde{\lambda}_{h}^{n+1},\nabla\cdot\chi_{h})
=(\nabla\times \bfg(t^{n+1}),\chi_{h}),\label{BE3}\\
(\nabla\cdot \tilde{\bB_{h}}^{n+1},r_{h}) = 0, \label{BE4}
\end{eqnarray}
{\bf Step 2:}
\begin{eqnarray}
\bu_{h}^{n+1}& =&\tilde {\bu}_h^{n+1}-\frac{1}{3}(\tilde {\bu}_h^{n+1}-2\bu_h^n+\bu_h^{n-1}),\label{stpu} \\
\bB_{h}^{n+1} &=&\tilde{\bB}_h^{n+1}-\frac{1}{3}(\tilde{\bB}_h^{n+1}-2\bB_h^n+\bB_h^{n-1}), \\
P_{h}^{n+1} &=&\tilde{P}_h^{n+1}-\frac{1}{3}(\tilde{P}_h^{n+1}-2P_h^n+P_h^{n-1}), \\
\lambda_{h}^{n+1} &=&\tilde{\lambda}_h^{n+1}-\frac{1}{3}(\tilde{\lambda}_h^{n+1}-2{\lambda}_h^n+{\lambda}_h^{n-1}), \label{stpL}
\end{eqnarray}
for all $(\bv_{h},\chi_{h},q_{h},r_{h}) \in
(\bX_{h},\bX_{h},Q_{h},Q_{h})$.
\end{algorithm}
Step 1, without Step 2, is the classical backward Euler scheme for MHD equations analyzed in \cite{CASE10}.
The numerical efficiency of the method is obvious. Step 2 is just an application of time filters as a modular step and its implementation is easy. 

By using the following operator, Step 2 can be embedded into Step 1 in the following way. Define the interpolation operator $\mathcal{F}$ as
\begin{eqnarray}
	\mathcal{F}[\bw_h^{n+1}]=\frac{3}{2} {\bw}_h^{n+1}-\bw_h^n+\frac{1}{2}\bw_h^{n-1}\label{intp}
\end{eqnarray}
which is formally $\mathcal{F}[\bw_h^{n+1}]=\bw_h^{n+1}+O(\Delta t^2)$. Note that reorganizing (\ref{stpu}) gives $\tilde {\bu}_{h}^{n+1}=\frac{3}{2} {\bu}_h^{n+1}-\bu_h^n+\frac{1}{2}\bu_h^{n-1}$. If one repeats same calculations for the other variables, inserting all of them in (\ref{BE1})-(\ref{BE4}) along with (\ref{intp}) gives
\begin{eqnarray}
\lefteqn{\frac{1}{\Delta t}(\frac{3}{2} {\bu}_h^{n+1}-2\bu_h^n+\frac{1}{2}\bu_h^{n-1},v_{h}) +Re^{-1}(\nabla(\mathcal{F}[\bu_h^{n+1}]),\nabla \bv_{h})
}\nonumber\\&&+ (\mathcal{F}[\bu_h^{n+1}]\cdot \nabla(\mathcal{F}[\bu_h^{n+1}]),\bv_{h})
-s(\mathcal{F}[\bB_h^{n+1}]\cdot \nabla(\mathcal{F}[\bB_h^{n+1}]),\bv_{h})\nonumber\\
&&-(\mathcal{F}[P_h^{n+1}],\nabla \cdot \bv_{h})
=(\bff(t^{n+1}),\bv_{h}),\label{sBE1}\\[9pt]
\lefteqn{(\nabla\cdot (\mathcal{F}[\bu_h^{n+1}]),q_{h})= 0,}\label{sBE2} \\[9
pt]
\lefteqn{\frac{1}{\Delta t}(\frac{3}{2}{\bB}_h^{n+1}-2\bB_h^n+\frac{1}{2}\bB_h^{n-1},\chi_{h})
	+Re_{m}^{-1}(\nabla \mathcal{F}[\bB_h^{n+1}],\nabla\chi_{h})}\nonumber\\
&&-(\mathcal{F}[\bB_h^{n+1}]\cdot \nabla(\mathcal{F}[\bu_h^{n+1}]),\chi_{h})+(\mathcal{F}[\bu_h^{n+1}]\cdot\nabla(\mathcal{F}[\bB_h^{n+1}]),\chi_{h})\nonumber\\
&&
+(\mathcal{F}[{\lambda}_{h}^{n+1}],\nabla\cdot\chi_{h})
=(\nabla\times \bfg(t^{n+1}),\chi_{h}),\label{sBE3}\\[9pt]
\lefteqn{(\nabla\cdot (\mathcal{F}[\bB_h^{n+1}]),r_{h}) = 0,} \label{sBE4}
\end{eqnarray}
for all $(\bv_{h},\chi_{h},q_{h},r_{h}) \in
(\bX_{h},\bX_{h},Q_{h},Q_{h})$. Naturally, the formulations (\ref{BE1})-(\ref{BE4}) and (\ref{sBE1})-(\ref{sBE4}) are equivalent. For simplicity of analysis, the equivalent formulation (\ref{sBE1})-(\ref{sBE4}) of the method will be used for the complete stability and convergence analysis. However, the utilization of the method for computer simulations will be based on (\ref{BE1})-(\ref{BE4}).

In the analysis, we use the following $G$-norm and $F$-norm. In general since $G$- stability implies $A$-stability, the use of $G$-matrix is very common in  BDF2 analysis, see e.g.,\cite{HW02} and references therein. These norms and properties are already have been given in \cite{JMR}.  With respect to notation of \cite{JMR},(see page 392),  analysis of the described method here uses the choices of $\theta=1$ and $\nu=2\epsilon$,
$$G=\begin{pmatrix}
 	\frac{3}{2}& 	-\frac{3}{4} \\
	-\frac{3}{4} &	\frac{1}{2}
 	\end{pmatrix},
$$
and the $G$-norm is given by
\begin{eqnarray}
\norm{ \begin{bmatrix}
 	\bu\\
 	\bv
 	\end{bmatrix}}_{G}^{2} =\big( \begin{bmatrix}
 \bu\\
 \bv
 \end{bmatrix},G \begin{bmatrix}
 \bu\\
 \bv
 \end{bmatrix}\big), \label{Gnorm}
 \end{eqnarray}
which can be negative.
Here  $\begin{bmatrix}
 \bu\\
 \bv
\end{bmatrix} $ is a $2n$ vector.

We also consider  $F=3I^n$ $\in$ $\mathbb{R}^{n\times n}$  symmetric positive matrix in general case, see \cite{JMR} for details. For any $\bu\in \mathbb{R}^{n}$, define $F$ norm of the $n$ vector $\bu$ by
\begin{eqnarray}
\norm{\bu}_{F}=(\bu,F\bu). \label{Fnorm}
\end{eqnarray}

The following properties of $G$-norm are well known and for a detailed derivation of these estimations, the reader is referred to {\cite{HW02,JMR}}.
\begin{lemma} \label{lemm:eqv}$L^2$ norm and $G$-norm are equivalent in the following sense: there exist constants $C_1 >C_2>0$ such that
\begin{eqnarray}
C_1	\norm{ \begin{bmatrix}
		\bu\\
		\bv
		\end{bmatrix}}_{G}^{2}\leq \norm{ \begin{bmatrix}
		\bu\\
		\bv
		\end{bmatrix}}^{2}\leq C_2\norm{ \begin{bmatrix}
		\bu\\
		\bv
		\end{bmatrix}}_{G}^{2}.
\end{eqnarray}	
\end{lemma}

\begin{lemma}\label{lem:inn}
The symmetric positive matrix  $F$ $\in$  $ \mathbb{R}^{n\times n}$ and the symmetric matrix $G$ $\in$
$ \mathbb{R}^{2n\times 2n}$ satisfy the following equality:
\begin{eqnarray}
\lefteqn{\Big(\frac{\frac{3}{2}w^{n+1}-2w^n+\frac {1}{2}w^{n-1}}{\Delta t} , \frac{3}{2} w^{n+1}- w^n+\frac{1}{2} w^{n-1}\Big) }\nonumber\\
 &=&\frac{1}{\Delta t}\norm{ \begin{bmatrix}
 	w^{n+1} \\
 	w^{n}\quad
 	\end{bmatrix}} _{G}^{2} -\frac{1}{\Delta t}\norm { \begin{bmatrix}
 	w^{n}\quad \\
 	w^{n-1}
 	\end{bmatrix}}_{G}^{2} 	+\dfrac{1}{4\Delta t}\norm{w^{n+1}-2w^{n}+w^{n-1}}_{F}^{2}. \label{innpro}
\end{eqnarray}
\end{lemma}

\begin{lemma}\label{lem:gnorm}
For any $\bu, \bv \in \mathbb{R}^{n}$, we have
\begin{eqnarray} 	
\big( \begin{bmatrix}
\bu\\
\bv
\end{bmatrix},G \begin{bmatrix}
\bu\\
\bv
\end{bmatrix}\big)
&\geq&\dfrac{3}{4}\norm{\bu}^2-\dfrac{1}{4}\norm{\bv}^2 ,\label{first}
\end{eqnarray}
\begin{eqnarray} 	
\big( \begin{bmatrix}
\bu\\
\bv
\end{bmatrix},G \begin{bmatrix}
\bu\\
\bv
\end{bmatrix}\big)&\leq&\frac{3}{2}\norm{\bu}^2+\frac{3}{4}\norm{\bv}^2.
\end{eqnarray}
\end{lemma}
\begin{proof}
Letting  $\theta=1$ and $\nu=2\epsilon$ in Lemma 3.1 on p. 392 of \cite{JMR} gives the stated result.
\end{proof}
The following consistency error estimations are required in the analysis.
\begin{lemma}\label{lem:fw} There exists $C>0$ such that
\begin{eqnarray}
\Delta t \sum_{n=1}^{N-1}\| \mathcal{F}[\bw^{n+1}]-\bw^{n+1}\|^2&\leq& C\Delta t^4 \| \bw_{tt}\|^2_{L^2(0,T;L^2(\Omega))},\label{dfw}\\
\Delta t \sum_{n=1}^{N-1}\|\frac{{3} {\bw}^{n+1}-4\bw^n+\bw^{n-1}}{2\Delta t}-\bw_t^{n+1}\|^2&\leq& C\Delta t^4\|\bw_{ttt}\|_{L^2(0,T;L^2(\Omega))}^2.\label{ww}
\end{eqnarray}
\end{lemma}
\begin{proof}
	Utilizing integral version of Taylor's theorem, we have
	\begin{eqnarray}
	 \mathcal{F}[\bw^{n+1}]-\bw^{n+1}&=&\frac{1}{2}\bw^{n+1}-\bw^n+\frac{1}{2}\bw^{n-1}\nonumber\\
		&=&\frac{1}{2}\Big(\bw^{n}+\Delta t \bw_t^{n}+\int_{t^{n}}^{t^{n+1}}\bw_{tt}(t^{n+1}-t) dt\Big)-\bw^n\nonumber\\
		&&+\frac{1}{2}\Big(\bw^{n}-\Delta t \bw_t^{n}+\int_{t^{n}}^{t^{n-1}}\bw_{tt}(t^{n-1}-t) dt\Big)\nonumber\\
		&\leq&C\Big(\int^{t^{n+1}}_{t^{n}}\bw_{tt}(t^{n+1}-t) dt+\int_{t^{n}}^{t^{n-1}}\bw_{tt}(t^{n-1}-t) dt\Big).
	\end{eqnarray}
	Hence, we get
	\begin{eqnarray}
		\lefteqn{\Big( \mathcal{F}[\bw^{n+1}]-\bw^{n+1}\Big)^2}\nonumber\\&\leq& C\Big(\int_{t^n}^{t^{n+1}} \bw_{tt}^2dt\int_{t^n}^{t^{n+1}}(t^n-t)^2dt+\int_{t^{n-1}}^{t^{n}}\bw_{tt}^2dt\int_{t^{n-1}}^{t^{n}}(t^{n-1}-t)^2dt\Big)\nonumber\\
		&\leq&C\Delta t^3\int_{t^{n-1}}^{t^{n+1}}\bw_{tt}^2dt.\label{dfw1}
	\end{eqnarray}
	In a similar manner, one gets
	\begin{eqnarray}
	\Big({	\frac{{3} {\bw}^{n+1}-4\bw^n+\bw^{n-1}}{2\Delta t}-\bw_t^{n+1}}\Big)^2\leq C\Delta t^3\int_{t^{n-1}}^{t^{n+1}} \bw_{ttt}^2dt.\label{ww1}
	\end{eqnarray}
	Integrating \eqref{dfw1} and \eqref{ww1} with respect to $x$ produces
		\begin{eqnarray}
	\norm{\mathcal{F}[\bw^{n+1}]-\bw^{n+1}}^2 
	&\leq&C\Delta t^3\int_{t^{n-1}}^{t^{n+1}}\|\bw_{tt}\|^2dt,\label{ss1} \\
	\norm{\frac{{3} {\bw}^{n+1}-4\bw^n+\bw^{n-1}}{2\Delta t}-\bw_t^{n+1}}^2&\leq& C\Delta t ^3\int_{t^{n-1}}^{t^{n+1}}\|\bw_{ttt}\|^2dt. \label{ss2}
	\end{eqnarray}
	Multiplying by $\Delta t$ and summing from $1$ to $N-1$ gives \eqref{dfw} and \eqref{ww}.
\end{proof}

\subsection{Conservation Laws}\label{cons}

We study conservation properties of the scheme \eqref{sBE1}-\eqref{sBE4}. Energy and helicity are very important flow quantities and play an important role in flow's structures, \cite {LM09}. It is well known that, an accurate model must predict these quantities correctly to verify the physical fidelity of the model. We now show the time filtered backward Euler \eqref{sBE1}-\eqref{sBE4} is an energy and helicity preserving scheme.

\begin{lemma}(Global Energy Conservation)
Scheme (\ref{sBE1})-(\ref{sBE4}) satisfies the energy equality:
\begin{eqnarray}
\lefteqn{\norm{ \begin{bmatrix}
		{\bu}_h^{N}\\[3pt]
		{\bu}_h^{N-1}
		\end{bmatrix}} _{G}^{2}  +s\norm{ \begin{bmatrix}
		{\bB}_h^{N} \\[3pt]
		{\bB}_h^{N-1}
		\end{bmatrix}} _{G}^{2} +\Delta t\sum_{n=1}^{N-1}\Big(Re^{-1}\|\nabla\mathcal{F}[ u_h^{n+1}]\|^2+sRe_{m}^{-1}\|\nabla \mathcal{F}[\bB_h^{n+1}]\|^2\Big)}\nonumber\\	 &&+\dfrac{1}{4}\sum_{n=1}^{N-1}\Big(\norm{{\bu}_h^{n+1}-2{\bu}_h^{n}+{\bu}_h^{n-1}}_{F}^{2} 	+s\norm{{\bB}_h^{n+1}-2{\bB}_h^{n}+{\bB}_h^{n-1}}_{F}^{2}\Big)\nonumber\\[5pt]
&=&\norm { \begin{bmatrix}
	{\bu}_h^{1}\\[3pt]
	{\bu}_h^{0}
	\end{bmatrix}}_{G}^{2}+s\norm { \begin{bmatrix}
	{\bB}_h^{1} \\[3pt]
	{\bB}_h^{0}
	\end{bmatrix}}_{G}^{2}+\Delta t\sum_{n=1}^{N-1}\Big((\bff(t^{n+1}),\mathcal{F}[\bu_h^{n+1}])\nonumber\\&&+s(\nabla\times \bfg(t^{n+1}),\mathcal{F}[\bB_h^{n+1}])\Big).\label{sEE5}
\end{eqnarray}

\end{lemma}
\begin{proof}
Set $v_{h} =\mathcal{F}[\bu_h^{n+1}]$ in {(\ref{sBE1})}, $q_h=\mathcal{F}[P_h^{n+1}]$ in {(\ref{sBE2})} , $\chi_{h}=\mathcal{F}[\bB_h^{n+1}]$ in {(\ref{sBE3})} and $r_h=\mathcal{F}[\lambda_h^{n+1}]$ in {(\ref{sBE4})}, then
the trilinear terms $(\mathcal{F}[\bu_h^{n+1}]\cdot\nabla(\mathcal{F}[\bu_h^{n+1}]),\mathcal{F}[\bu_h^{n+1}])$ and $(\mathcal{F}[\bu_h^{n+1}]\cdot\nabla(\mathcal{F}[\bB_h^{n+1}]), \mathcal{F}[\bB_h^{n+1}])$, the pressure term and the $\lambda$ term vanish by the use of {(\ref{divuB})}. Then, one gets
\begin{eqnarray}
\lefteqn{\frac{1}{\Delta t}(\frac{3}{2} {\bu}_h^{n+1}-2\bu_h^n+\frac{1}{2}\bu_h^{n-1},\mathcal{F}[\bu_h^{n+1}]) +Re^{-1}\|\nabla\mathcal{F}[\bu_h^{n+1}]\|^2}
\nonumber\\
&&-s(\mathcal{F}[\bB_h^{n+1}]\cdot \nabla(\mathcal{F}[\bB_h^{n+1}]),\mathcal{F}[\bu_h^{n+1}])
=(\bff(t^{n+1}),\mathcal{F}[\bu_h^{n+1}]),\label{sE1}\\[7pt]
\lefteqn{\frac{1}{\Delta t}(\frac{3}{2}{\bB}_h^{n+1}-2\bB_h^n+\frac{1}{2}\bB_h^{n-1},\mathcal{F}[\bB_h^{n+1}])
	+Re_{m}^{-1}\|\nabla \mathcal{F}[\bB_h^{n+1}]\|^2}\nonumber\\
&&-(\mathcal{F}[\bB_h^{n+1}]\cdot\nabla( \mathcal{F}[\bu_h^{n+1}]),\mathcal{F}[\bB_h^{n+1}])
=(\nabla\times \bfg(t^{n+1}),\mathcal{F}[\bB_h^{n+1}]).\label{sE3}
\end{eqnarray}
Note that since $\nabla \cdot \mathcal{F}[B_h^{n+1}]=0$, we get $(\mathcal{F}[\bB_h^{n+1}]\cdot\nabla(\mathcal{F}[\bB_h^{n+1}]),\mathcal{F}[\bu_h^{n+1}])=-(\mathcal{F}[\bB_h^{n+1}]\cdot\nabla( \mathcal{F}[\bu_h^{n+1}]),\mathcal{F}[\bB_h^{n+1}])$. Multiplying (\ref{sE3}) by $s$ and adding (\ref{sE1}) to (\ref{sE3}) produces
 \begin{eqnarray}
\lefteqn{\frac{1}{\Delta t}(\frac{3}{2} {\bu}_h^{n+1}-2\bu_h^n+\frac{1}{2}\bu_h^{n-1},\mathcal{F}[\bu_h^{n+1}]) +Re^{-1}\|\nabla\mathcal{F}[\bu_h^{n+1}]\|^2}\nonumber\\
&&+\frac{1}{\Delta t}s(\frac{3}{2}{\bB}_h^{n+1}-2\bB_h^n+\frac{1}{2}\bB_h^{n-1},\mathcal{F}[\bB_h^{n+1}])
+sRe_{m}^{-1}\|\nabla( \mathcal{F}[\bB_h^{n+1}])\|^2\nonumber\\
&=&(\bff(t^{n+1}),\mathcal{F}[\bu_h^{n+1}])+s(\nabla\times \bfg(t^{n+1}),\mathcal{F}[\bB_h^{n+1}]).
\label{sE4}
\end{eqnarray}
Reorganizing (\ref{sE4}) by using Lemma \ref{lem:inn} and multiplying with $\Delta t$ yields
\begin{eqnarray}
\lefteqn{\norm{ \begin{bmatrix}
		{\bu}_h^{n+1}\\
		{\bu}_h^{n}
		\end{bmatrix}} _{G}^{2} -\norm { \begin{bmatrix}
		{\bu}_h^{n}\\
		{\bu}_h^{n-1}
		\end{bmatrix}}_{G}^{2} +s\norm{ \begin{bmatrix}
		{\bB}_h^{n+1} \\
		{\bB}_h^{n}
		\end{bmatrix}} _{G}^{2} -s\norm { \begin{bmatrix}
		{\bB}_h^{n} \\
		{\bB}_h^{n-1}
		\end{bmatrix}}_{G}^{2}}\nonumber\\[5pt]	&&+Re^{-1}\Delta t\|\nabla( \mathcal{F}[\bu_h^{n+1}])\|^2+sRe_{m}^{-1}\Delta t\|\nabla( \mathcal{F}[\bB_h^{n+1}])\|^2\nonumber\\
&& 	+\dfrac{1}{4}\norm{{\bu}_h^{n+1}-2{\bu}_h^{n}+{\bu}_h^{n-1}}_{F}^{2}+\dfrac{s}{4}\norm{{\bB}_h^{n+1}-2{\bB}_h^{n}+{\bB}_h^{n-1}}_{F}^{2}\nonumber\\[5pt]
&=&\Delta t(\bff(t^{n+1}),\mathcal{F}[\bu_h^{n+1}])+s\Delta t(\nabla\times \bfg(t^{n+1}),\mathcal{F}[\bB_h^{n+1}]).\label{sEE4}
\end{eqnarray}
Summing (\ref{sEE4}) from $n=1$ to $n=N-1$ gives the stated energy result.
\end{proof}
\begin{lemma}(Global Cross Helicity Conservation)
Scheme (\ref{sBE1})-(\ref{sBE4}) satisfies the cross helicity equality:
\begin{eqnarray}
\lefteqn{\Big( \begin{bmatrix}
\bu_h^{N}\quad\\[3pt]
\bu_h^{N-1}
\end{bmatrix},G \begin{bmatrix}
\bB_h^{N}\quad\\[3pt]
\bB_h^{N-1}
\end{bmatrix}\Big)+\Big( \begin{bmatrix}
\bB_h^{N}\quad\\[3pt]
\bB_h^{N-1}
\end{bmatrix},G \begin{bmatrix}
\bu_h^{N}\quad\\[3pt]
\bu_h^{N-1}
\end{bmatrix}\Big)}\nonumber\\
&&+{\dfrac{1}{2}\sum_{n=1}^{N-1}(I[\bu_h^{n+1}],I[\bB_h^{n+1}])}\nonumber\\
&&+\Delta t\sum_{n=1}^{N-1}(Re^{-1}+Re_{m}^{-1})(\nabla\mathcal{F}[\bu_h^{n+1}],\nabla \mathcal{F}[\bB_h^{n+1}])\nonumber
\\&=&\Big( \begin{bmatrix}
\bu_h^{1}\\[3pt]
\bu_h^{0}
\end{bmatrix},G \begin{bmatrix}
\bB_h^{1}\\[3pt]
\bB_h^{0}
\end{bmatrix}\Big)+\Big( \begin{bmatrix}
\bB_h^{1}\\[3pt]
\bB_h^{0}
\end{bmatrix},G \begin{bmatrix}
\bu_h^{1}\\[3pt]
\bu_h^{0}
\end{bmatrix}\Big)\nonumber\\&&+\Delta t\Big(\sum_{n=1}^{N-1}(\bff(t^{n+1}),\mathcal{F}[\bB_h^{n+1}])+(\nabla\times \bfg(t^{n+1}),\mathcal{F}[\bu_h^{n+1}])\Big)\label{crosh}
\end{eqnarray}
where $I[\bw_h^{n+1}]=\bw_h^{n+1}-2\bw_h^n+\bw_h^{n-1}$
\end{lemma}
\begin{proof}
To prove the global cross helicity conservation, set $v_{h} =\mathcal{F}[\bB_h^{n+1}]$, $q_{h} =\mathcal{F}[\lambda_h^{n+1}]$,  $\chi_{h}=\mathcal{F}[\bu_h^{n+1}]$, $r_{h} =\mathcal{F}[P_h^{n+1}]$ in {(\ref{sBE1})}- {(\ref{sBE4})}, respectively. Since the trilinear terms $(\mathcal{F}[\bB_h^{n+1}]\cdot \nabla(\mathcal{F}[\bB_h^{n+1}]),\mathcal{F}[\bB_h^{n+1}])$, $(\mathcal{F}[\bB_h^{n+1}]\cdot\nabla(\mathcal{F}[\bu_h^{n+1}]),\mathcal{F}[\bu_h^{n+1}])$, the pressure term and the $\lambda$ term vanish by the use of \eqref{divuB} , one has
\begin{eqnarray}
\lefteqn{\frac{1}{\Delta t}(\frac{3}{2} {\bu}_h^{n+1}-2\bu_h^n+\frac{1}{2}\bu_h^{n-1},\mathcal{F}[\bB_h^{n+1}]) +Re^{-1}(\nabla\mathcal{F}[\bu_h^{n+1}],\nabla \mathcal{F}[\bB_h^{n+1}])}\nonumber\\&&+ (\mathcal{F}[\bu_h^{n+1}]\cdot \nabla\mathcal{F}[\bu_h^{n+1}],\mathcal{F}[\bB_h^{n+1}])
=(\bff(t^{n+1}),\mathcal{F}[\bB_h^{n+1}]),\label{sCE1}\\[9pt]
\lefteqn{\frac{1}{\Delta t}(\frac{3}{2}{\bB}_h^{n+1}-2\bB_h^n+\frac{1}{2}\bB_h^{n-1},\mathcal{F}[\bu_h^{n+1}])
	+Re_{m}^{-1}(\nabla \mathcal{F}[\bB_h^{n+1}],\nabla\mathcal{F}[\bu_h^{n+1}])}\nonumber\\
&&+(\mathcal{F}[\bu_{h}^{n+1}]\cdot \nabla \mathcal{F}[\bB_h^{n+1}],\mathcal{F}[\bu_h^{n+1}])
=(\nabla\times \bfg(t^{n+1}),\mathcal{F}[\bu_h^{n+1}]).\label{sCE2}
\end{eqnarray}
Note that since  {$(\mathcal{F}[\bu_h^{n+1}]\cdot \nabla(\mathcal{F}[\bu_h^{n+1}]),\mathcal{F}[\bB_h^{n+1}])=-(\mathcal{F}[\bu_h^{n+1}]\cdot\nabla( \mathcal{F}[\bB_h^{n+1}]),\mathcal{F}[\bu_h^{n+1}])$} and
\begin{eqnarray}
\lefteqn{\frac{1}{\Delta t}(\frac{3}{2} {\bu}_h^{n+1}-2\bu_h^n+\frac{1}{2}\bu_h^{n-1},\mathcal{F}[\bB_h^{n+1}])+\frac{1}{\Delta t}(\frac{3}{2} {\bB}_h^{n+1}-2\bB_h^n+\frac{1}{2}\bB_h^{n-1},\mathcal{F}[\bu_h^{n+1}])}\nonumber\\
&=&\dfrac{1}{\Delta t}\Bigg(\Big( \begin{bmatrix}
\bu_h^{n+1}\\
\bu_h^n\quad
\end{bmatrix},G \begin{bmatrix}
\bB_h^{n+1}\\
\bB_h^n\quad
\end{bmatrix}\Big)-\Big( \begin{bmatrix}
\bu_h^{n}\quad\\
\bu_h^{n-1}
\end{bmatrix},G \begin{bmatrix}
\bB_h^{n}\quad\\
\bB_h^{n-1}
\end{bmatrix}\Big)\Bigg)\nonumber\\
&&+\dfrac{1}{\Delta t}\Bigg(\Big( \begin{bmatrix}
\bB_h^{n+1}\\
\bB_h^n\quad
\end{bmatrix},G \begin{bmatrix}
\bu_h^{n+1}\\
\bu_h^n\quad
\end{bmatrix}\Big)-\Big( \begin{bmatrix}
\bB_h^{n}\quad\\
\bB_h^{n-1}
\end{bmatrix},G \begin{bmatrix}
\bu_h^{n}\quad\\
\bu_h^{n-1}
\end{bmatrix}\Big)\Bigg)\nonumber\\
&&+\dfrac{1}{2\Delta t}(I[\bu_h^{n+1}],I[\bB_h^{n+1}]),\label{timeG}
\end{eqnarray}
where $I[\bw_h^{n+1}]=\bw_h^{n+1}-2\bw_h^n+\bw_h^{n-1}$, adding (\ref{sCE1}) and (\ref{sCE2}) yields
\begin{eqnarray}
&&\dfrac{1}{\Delta t}\Big( \begin{bmatrix}
\bu_h^{n+1}\\
\bu_h^n\quad
\end{bmatrix},G \begin{bmatrix}
\bB_h^{n+1}\\
\bB_h^n\quad
\end{bmatrix}\Big)+\dfrac{1}{\Delta t}\Big( \begin{bmatrix}
\bB_h^{n+1}\\
\bB_h^n\quad
\end{bmatrix},G \begin{bmatrix}
\bu_h^{n+1}\\
\bu_h^n\quad
\end{bmatrix}\Big)\nonumber\\
&&+\dfrac{1}{2\Delta t}(I[\bu_h^{n+1}],I[\bB_h^{n+1}])\nonumber\\
&&+Re^{-1}(\nabla\mathcal{F}[\bu_h^{n+1}],\nabla \mathcal{F}[\bB_h^{n+1}])
+Re_{m}^{-1}(\nabla \mathcal{F}[\bB_h^{n+1}],\nabla\mathcal{F}[\bu_h^{n+1}])\nonumber
\\&&=\dfrac{1}{\Delta t}\Big( \begin{bmatrix}
\bu_h^{n}\quad\\
\bu_h^{n-1}
\end{bmatrix},G \begin{bmatrix}
\bB_h^{n}\quad\\
\bB_h^{n-1}
\end{bmatrix}\Big)+\dfrac{1}{\Delta t}\Big( \begin{bmatrix}
\bB_h^{n}\quad\\
\bB_h^{n-1}
\end{bmatrix},G \begin{bmatrix}
\bu_h^{n}\quad\\
\bu_h^{n-1}
\end{bmatrix}\Big)\nonumber\\&&+(\bff(t^{n+1}),\mathcal{F}[\bB_h^{n+1}])+(\nabla\times \bfg(t^{n+1}),\mathcal{F}[\bu_h^{n+1}]).\label{sCE3}
\end{eqnarray}
Summing (\ref{sCE3}) from $n=1$ to $n=N-1$ and multiplying by $\Delta t$ produces the cross helicity conservation result.
\end{proof}

\section{Convergence Analysis}

\subsection {Stability and Long Time Stability }
This section presents unconditional stability, long time stability and convergence analysis of the proposed method.

\begin{lemma} \label{l:stab}
Let $f\in L^2(0,T; H^{-1}(\Omega))$ and $g\in L^2(0,T; L^2(\Omega))$. Then, solutions
to the scheme (\ref{sBE1})-(\ref{sBE4})
are unconditionally stable, and satisfy the following bounds at $t_{M}=M\Delta t$
\begin{eqnarray}
\lefteqn{\|{\bu}_h^{N}\|^2	+s\|{\bB}_h^{N}\|^2
+\frac{2\Delta t}{3}\sum_{n=1}^{N-1}\Big(Re^{-1}\|\nabla ( \mathcal{F}[\bu_h^{n+1}])\|^2 +sRe_{m}^{-1}\|\nabla \mathcal{F}[\bB_h^{n+1}]\|^2\Big)}\nonumber\\
&&	+\dfrac{1}{3}\sum_{n=1}^{N-1}\Big(\norm{{\bu}_h^{n+1}-2{\bu}_h^{n}+{\bu}_h^{n-1}}_{F}^{2}
+s\norm{{\bB}_h^{n+1}-2{\bB}_h^{n}+{\bB}_h^{n-1}}_{F}^{2}\Big)\nonumber\nonumber\\[5pt]
&\leq&\Big(\frac{1}{3}\Big)^N(\|{\bu}_h^{0}\|^2+s\|{\bB}_h^{0}\|^2)+N\Big( 2(\|{\bu}_h^{1}\|^2+s\|{\bB}_h^{1}\|^2)+(\|{\bu}_h^{0}\|^2+s\|{\bB}_h^{0}\|^2) \Big) \nonumber\\
&&+\frac{2 N\Delta t}{3}\sum_{n=1}^{N-1}\Big(Re\|\bff(t^{n+1})\|_{-1}^2+s Re_m\|\bfg(t^{n+1})\|^2\Big).\label{sE8}
\end{eqnarray}
\end{lemma}
\begin{proof}
The proof starts with using the global energy conservation (\ref{sEE5}).

\begin{eqnarray}
\lefteqn{\norm{ \begin{bmatrix}
		{\bu}_h^{N}\\[3pt]
		{\bu}_h^{N-1}
		\end{bmatrix}} _{G}^{2}  +s\norm{ \begin{bmatrix}
		{\bB}_h^{N} \\[3pt]
		\quad{\bB}_h^{N-1}
		\end{bmatrix}} _{G}^{2} +\Delta t\sum_{n=1}^{N-1}\Big(Re^{-1}\|\nabla\mathcal{F}[\bu_h^{n+1}]\|^2+sRe_{m}^{-1}\|\nabla \mathcal{F}[\bB_h^{n+1}]\|^2\Big)}\nonumber\\	 &&+\dfrac{1}{4}\sum_{n=1}^{N-1}\Big(\norm{{\bu}_h^{n+1}-2{\bu}_h^{n}+{\bu}_h^{n-1}}_{F}^{2} 	+s\norm{{\bB}_h^{n+1}-2{\bB}_h^{n}+{\bB}_h^{n-1}}_{F}^{2}\Big)\nonumber\\[5pt]
&=&\norm { \begin{bmatrix}
	{\bu}_h^{1}\\[3pt]
	{\bu}_h^{0}
	\end{bmatrix}}_{G}^{2}+s\norm { \begin{bmatrix}
	{\bB}_h^{1} \\[3pt]
	{\bB}_h^{0}
	\end{bmatrix}}_{G}^{2}+\Delta t\sum_{n=1}^{N-1}(\bff(t^{n+1}),\mathcal{F}[\bu_h^{n+1}])\nonumber\\
&&+\Delta t\sum_{n=1}^{N-1}s(\nabla\times \bfg(t^{n+1}),\mathcal{F}[\bB_h^{n+1}]).\label{sEE6}
\end{eqnarray}
The forcing terms  can be bounded by using Cauchy-Schwarz and Young's inequalities as
\begin{eqnarray}
\Delta t(\bff(t^{n+1}),\mathcal{F}[\bu_h^{n+1}])&\leq&\frac{Re \Delta t}{2}\|\bff(t^{n+1})\|_{-1}^2+\frac{Re^{-1}\Delta t}{2}\|\nabla\mathcal{F}[\bu_h^{n+1}]\|^2,\label{su}\\[5pt]
s\Delta t(\nabla\times \bfg(t^{n+1}),\mathcal{F}[\bB_h^{n+1}])&\leq& \frac{s Re_m\Delta t}{2}\| \bfg(t^{n+1})\|^2\nonumber\\&&+\frac{s Re_m^{-1}\Delta t}{2}\|\nabla\mathcal{F}[\bB_h^{n+1}]\|^2.\label{sb}
\end{eqnarray}
Inserting (\ref{su}) and (\ref{sb}) in (\ref{sEE6}) gives
\begin{eqnarray}
\lefteqn{\norm{ \begin{bmatrix}
		{\bu}_h^{N}\\[3pt]
		{\bu}_h^{N-1}
		\end{bmatrix}} _{G}^{2} 	+s\norm{ \begin{bmatrix}
		{\bB}_h^{N} \\[3pt]
		{\bB}_h^{N-1}
		\end{bmatrix}} _{G}^{2}	+\frac{\Delta t}{2}\sum_{n=1}^{N-1}\Big(Re^{-1}\|\nabla\mathcal{F}[\bu_h^{n+1}]\|^2+{sRe_{m}^{-1}}\|\nabla \mathcal{F}[\bB_h^{n+1}]\|^2\Big)}\nonumber\\&&+\dfrac{1}{4}\sum_{n=1}^{N-1}\Big(\norm{{\bu}_h^{n+1}-2{\bu}_h^{n}+{\bu}_h^{n-1}}_{F}^{2}+s\norm{{\bB}_h^{n+1}-2{\bB}_h^{n}+{\bB}_h^{n-1}}_{F}^{2}\Big) \nonumber\\[5pt]
&\leq&\norm { \begin{bmatrix}
	{\bu}_h^{1}\\[3pt]
	{\bu}_h^{0}
	\end{bmatrix}}_{G}^{2}+s\norm { \begin{bmatrix}
	{\bB}_h^{1} \\[3pt]
	{\bB}_h^{0}
	\end{bmatrix}}_{G}^{2}  +\frac{Re \Delta t}{2}\sum_{n=1}^{N-1}\|\bff(t^{n+1})\|_{-1}^2\nonumber\\
&&+\frac{s Re_m\Delta t}{2}\sum_{n=1}^{N-1}\|\bfg(t^{n+1})\|^2.\label{sE6}
\end{eqnarray}
Using Lemma \ref{lem:gnorm}, we get
\begin{eqnarray}
\lefteqn{\frac{3}{4}(\|{\bu}_h^{N}\|^2+s\|{\bB}_h^{N}\|^2)
			+\frac{\Delta t}{2}\sum_{n=1}^{N-1}\Big(Re^{-1}\|\nabla\mathcal{F}[\bu_h^{n+1}]\|^2+{sRe_{m}^{-1}}\|\nabla \mathcal{F}[\bB_h^{n+1}]\|^2\Big)}\nonumber\\&&+\dfrac{1}{4}\sum_{n=1}^{N-1}\Big(\norm{{\bu}_h^{n+1}-2{\bu}_h^{n}+{\bu}_h^{n-1}}_{F}^{2}+s\norm{{\bB}_h^{n+1}-2{\bB}_h^{n}+{\bB}_h^{n-1}}_{F}^{2}\Big)\nonumber\\
&\leq&\frac{1}{4}(\|{\bu}_h^{N-1}\|^2+s\|{\bB}_h^{N-1}\|^2)+\frac{3}{2}(\|{\bu}_h^{1}\|^2+s\|{\bB}_h^{1}\|^2)+\frac{3}{4}(\|{\bu}_h^{0}\|^2+s\|{\bB}_h^{0}\|^2) \nonumber\\
&& +\frac{Re \Delta t}{2}\sum_{n=1}^{N-1}\|\bff(t^{n+1})\|_{-1}^2+\frac{s Re_m\Delta t}{2}\sum_{n=1}^{N-1}\|\bfg(t^{n+1})\|^2.\label{sE7}
\end{eqnarray}
Multiplying (\ref{sE7}) by $\frac{4}{3}$ and using induction finishes the proof.
\end{proof}
We also show that the scheme is unconditionally long
time stable.
\begin{lemma} 
	 Let  $\bff, \bfg\in L^{\infty}(\mathbb{R}_{+},\bfV_h^*)$, then
	the approximation \eqref{sBE1}-\eqref{sBE4} is long time stable in the following sense: for any $\Delta t > 0$
  \begin{eqnarray}
\lefteqn{\Big(\norm{ \begin{bmatrix}
		{\bu}_h^{N}\quad\\[3pt]
		{\bu}_h^{N-1}
		\end{bmatrix}} _{G}^{2}+\dfrac{Re^{-1}\Delta t}{8}\|\nabla\mathcal{F}[\bu_h^{N}]\|^2 \Big)	+\Big(s\norm{ \begin{bmatrix}
		{\bB}_h^{N}\quad \\[3pt]
		{\bB}_h^{N-1}
		\end{bmatrix}} _{G}^{2}+\dfrac{sRe_m^{-1}\Delta t}{8}\|\nabla\mathcal{F}[\bB_h^{N}]\|^2\Big)}	\nonumber\\
&\leq&\omega^{-({n+1})}\Big(\norm { \begin{bmatrix}
	{\bu}_h^{1}\\[3pt]
	{\bu}_h^{0}
	\end{bmatrix}}_{G}^{2}+\dfrac{Re^{-1}\Delta t}{8}\|\nabla\mathcal{F}[\bu_h^{1}]\|^2\Big)+\omega^{-({n+1})}\Big(s\norm { \begin{bmatrix}
	{\bB}_h^{1} \\[3pt]
	{\bB}_h^{0}
	\end{bmatrix}}_{G}^{2} +\dfrac{sRe_m^{-1}\Delta t}{8}\|\nabla\mathcal{F}[\bB_h^{1}]\|^2\Big)\nonumber\\&& +\frac{Re \Delta t}{2\omega}\|\bff(t^{n+1})\|_{L^{\infty}(\mathbb{R}_{+},{\bfV_h^*})}^2+\frac{s Re_m\Delta t}{2\omega}\|\bfg(t^{n+1})\|_{L^{\infty}(\mathbb{R}_{+},{\bfV_h^*})}^2 ,\label{longstabl}
\end{eqnarray}
where $\omega=(1+\alpha)(1+\beta)$, $\alpha=\min\{\dfrac{C_1^2Re^{-1}\Delta t}{8C_p^2},2\}$,$\beta=\min\{\dfrac{C_1^2Re_m^{-1}\Delta t}{8C_p^2},2\}$, $C_1$ is given by Lemma \ref{lemm:eqv} and $C_p$ is given by \eqref{pof}.
\end{lemma}
\begin{proof}
	Applying Cauchy-Schwarz and Young's inequalities for the global energy conservation equation (\ref{sEE5}) yields
	\begin{eqnarray}
\lefteqn{\norm{ \begin{bmatrix}
		{\bu}_h^{n+1}\\[3pt]
		{\bu}_h^{n}\quad
		\end{bmatrix}} _{G}^{2} 	+s\norm{ \begin{bmatrix}
		{\bB}_h^{n+1} \\[3pt]
		{\bB}_h^{n}\quad
		\end{bmatrix}} _{G}^{2}	+\frac{\Delta t}{2}\Big(Re^{-1}\|\nabla\mathcal{F}[ \bu_h^{n+1}]\|^2+{sRe_{m}^{-1}}\|\nabla \mathcal{F}[\bB_h^{n+1}]\|^2\Big)}\nonumber\\&&+\dfrac{1}{4}\Big(\norm{{\bu}_h^{n+1}-2{\bu}_h^{n}+{\bu}_h^{n-1}}_{F}^{2}+s\norm{{\bB}_h^{n+1}-2{\bB}_h^{n}+{\bB}_h^{n-1}}_{F}^{2}\Big) \nonumber\\[5pt]
&\leq&\norm { \begin{bmatrix}
	{\bu}_h^{n}\quad\\[3pt]
	{\bu}_h^{n-1}
	\end{bmatrix}}_{G}^{2}+s\norm { \begin{bmatrix}
	{\bB}_h^{n}\quad \\[3pt]
	{\bB}_h^{n-1}
	\end{bmatrix}}_{G}^{2}  +\frac{Re \Delta t}{2}\|\bff(t^{n+1})\|_{\bfV_h^*}^2\nonumber\\
&&+\frac{s Re_m\Delta t}{2}\|\bfg(t^{n+1})\|_{\bfV_h^*}^2.\label{sE9}
\end{eqnarray}

 Dropping the fourth term in the left hand side of \eqref{sE9}, and adding both sides $\dfrac{Re^{-1}\Delta t}{8}\|\nabla\mathcal{F}[\bu_h^{n}]\|^2$ and $\dfrac{sRe_m^{-1}\Delta t}{8}\|\nabla\mathcal{F}[\bB_h^{n}]\|^2$ results in
 \begin{eqnarray}
 \lefteqn{\Big(\norm{ \begin{bmatrix}
 		{\bu}_h^{n+1}\\[3pt]
 		{\bu}_h^{n}\quad
 		\end{bmatrix}} _{G}^{2}+\dfrac{Re^{-1}\Delta t}{8}\|\nabla\mathcal{F}[\bu_h^{n+1}]\|^2 \Big)	+\Big(s\norm{ \begin{bmatrix}
 		{\bB}_h^{n+1} \\[3pt]
 		{\bB}_h^{n}\quad
 		\end{bmatrix}} _{G}^{2}+\dfrac{sRe_m^{-1}\Delta t}{8}\|\nabla\mathcal{F}[\bB_h^{n+1}]\|^2\Big)}	\nonumber\\[3pt]&&+\dfrac{Re^{-1}\Delta t}{8}(\|\nabla\mathcal{F}[\bu_h^{n+1}]\|^2+\|\nabla\mathcal{F}[\bu_h^{n}]\|^2)+\frac{Re^{-1}\Delta t}{4}\|\nabla\mathcal{F}[ \bu_h^{n+1}]\|^2\nonumber\\[3pt]&&+\dfrac{sRe_m^{-1}\Delta t}{8}(\|\nabla\mathcal{F}[\bB_h^{n+1}]\|^2+\|\nabla\mathcal{F}[\bB_h^{n}]\|^2)+\frac{sRe_{m}^{-1}\Delta t}{4}\|\nabla \mathcal{F}[\bB_h^{n+1}]\|^2\nonumber\\[5pt]
 &\leq&\Big(\norm { \begin{bmatrix}
 	{\bu}_h^{n}\quad\\[3pt]
 	{\bu}_h^{n-1}
 	\end{bmatrix}}_{G}^{2}+\dfrac{Re^{-1}\Delta t}{8}\|\nabla\mathcal{F}[\bu_h^{n}]\|^2\Big)+\Big(s\norm { \begin{bmatrix}
 	{\bB}_h^{n} \quad\\[3pt]
 	{\bB}_h^{n-1}
 	\end{bmatrix}}_{G}^{2} +\dfrac{sRe_m^{-1}\Delta t}{8}\|\nabla\mathcal{F}[\bB_h^{n}]\|^2\Big)\nonumber\\[3pt]&& +\frac{Re \Delta t}{2}\|\bff(t^{n+1})\|_{\bfV_h^*}^2+\frac{s Re_m\Delta t}{2}\|\bfg(t^{n+1})\|_{\bfV_h^*}^2.\label{sE10}
 \end{eqnarray}
 The third and fourth terms can be bounded by using Poincar\'e's-Friedrichs' inequality and Lemma \ref{lemm:eqv} as
 \begin{eqnarray}
 \lefteqn{\dfrac{Re^{-1}\Delta t}{8}(\|\nabla\mathcal{F}[\bu_h^{n+1}]\|^2+\|\nabla\mathcal{F}[\bu_h^{n}]\|^2)+\frac{Re^{-1}\Delta t}{4}\|\nabla\mathcal{F}[\bu_h^{n+1}]\|^2}\nonumber\\
 &\geq&\dfrac{C_1^2Re^{-1}\Delta t}{8C_p^2}\norm{ \begin{bmatrix}
 	{\bu}_h^{n+1}\\[3pt]
 	{\bu}_h^{n}\quad
 	\end{bmatrix}} _{G}^{2}+\frac{Re^{-1}\Delta t}{4}\|\nabla\mathcal{F}[ \bu_h^{n+1}]\|^2\nonumber\\
 &\geq&\alpha\Big(\norm{ \begin{bmatrix}
 	{\bu}_h^{n+1}\\[3pt]
 	{\bu}_h^{n}\quad
 	\end{bmatrix}} _{G}^{2}+\dfrac{Re^{-1}\Delta t}{8}\|\nabla\mathcal{F}[\bu_h^{n+1}]\|^2 \Big),\label{alp}
 \end{eqnarray}
 where $\alpha=\min\{\dfrac{C_1^2Re^{-1}\Delta t}{8C_p^2},2\}$. Using similar techniques for the fifth and the sixth terms,  we get
 \begin{eqnarray}
 \lefteqn{\dfrac{sRe_m^{-1}\Delta t}{8}(\|\nabla\mathcal{F}[\bB_h^{n+1}]\|^2+\|\nabla\mathcal{F}[\bB_h^{n}]\|^2)+\frac{sRe_{m}^{-1}\Delta t}{4}\|\nabla \mathcal{F}[\bB_h^{n+1}]\|^2}\nonumber\\&\geq&\beta \Big(s\norm{ \begin{bmatrix}
 	{\bB}_h^{n+1} \\[3pt]
 	{\bB}_h^{n}\quad
 	\end{bmatrix}} _{G}^{2}+\dfrac{sRe_m^{-1}\Delta t}{8}\|\nabla\mathcal{F}[\bB_h^{n+1}]\|^2\Big),\label{bet}
 \end{eqnarray}
where $\beta=\min\{\dfrac{C_1^2Re_m^{-1}\Delta t}{8C_p^2},2\}$. Inserting \eqref{alp}-\eqref{bet} in \eqref{sE10} and using induction yields the stated result.

\end{proof}

\subsection{A-{\it priori} Error Estimate}

In this section, we present  a detailed convergence analysis of the proposed time filtered method for MHD equations.
We define the discrete norms as
\begin{eqnarray}
|||w|||_{\infty,m}=\max\limits_{0\leq n \leq N}||w^n||_m,\quad  |||w|||_{k,m}=\Big(\Delta t \sum\limits_{n=1}^{N-1}||w^n||_m^k \Big)^{\frac{1}{k}}.
\end{eqnarray}
For the optimal asymptotic error estimation, we assume the following regularity assumptions for the exact solution of \eqref{mhd5}-\eqref{mhd8}:
\begin{eqnarray}
	&& u,\,B \in{L^\infty(0,T;H^{s+1}(\Omega) \cap H^3(\Omega) )},\nonumber\\
	&&u_t,\,B_t \in{L^2(0,T;H^{s+1}(\Omega)^d)},\label{reg}\\
	&&u_{tt},\,B_{tt}\in{L^2(0,T;H^{1}(\Omega)^d)},\nonumber\\
		&&u_{ttt},\,B_{ttt}\in{L^2(0,T;L^{2}(\Omega)^d)}.\nonumber
\end{eqnarray}
The mesh and velocity approximating polynomial degree $k$ is chosen so that the Scott-Vogelius pair is inf-sup stable and the properties \eqref{app1}-\eqref{app2} hold.
\begin{theorem}\label{The:conv} Suppose regularity assumptions \eqref{reg} hold. Under the following time step condition
\begin{eqnarray}
\Delta t &\leq&C(s)\Big( |||\nabla\bu|||_{\infty,0}^4+ ||| \nabla  B|||_{\infty,0}^4\Big)^{-1},
\end{eqnarray}
there exists a positive constant $C^*$ independent of $h$ and $\Delta t$ such that	
\begin{eqnarray}
\lefteqn{\norm{\bu^N-{\bu}_h^{N}}^{2}+s\norm{\bB^N-{\bB}_h^{N}}^{2}
	+\frac{\Delta t}{3}\sum_{n=0}^{N-1}\Big(Re^{-1}\|\nabla (\bu^{n+1}-{\bu}_h^{n+1})\|^{2}}\nonumber\\
	&&+sRe_m^{-1}\norm{\nabla (\bB^{n+1}-{\bB}^{n+1}_h)}^{2}\Big)\nonumber
\\
&\leq& \Big(\frac{1}{3}\Big)^N(\norm{{\bu}_h^0-I_{\bu}^0}^{2}+ s\norm{{\bB}_h^0-I_{\bB}^0}^{2})+2N(\norm{{\bu}_h^1-I_{\bB}^1}^{2}+ s\norm{{\bB}_h^1-I_{\bB}^1}^{2}\nonumber\\&&+\norm{{\bu}_h^0-I_{\bu}^0}+ s\norm{{\bB}_h^0-I_{\bB}^0}^{2})+C^*(
{h^{2s}}
+{\Delta t}^4)
. \label{thm}
\end{eqnarray}

\end{theorem}

\begin{proof} The proof starts by deriving the error equations. We
consider continuous variational formulations of
\eqref{mhd5}-\eqref{mhd8} at the time level $t=t^{n+1}$. Adding
and subtracting terms yields the following variational
formulations for the velocity,
\begin{eqnarray}
\lefteqn{\left(\frac{\frac{3}{2} {\bu}^{n+1}-2\bu^n+\frac{1}{2}\bu^{n-1}}{\Delta t},\bv_h\right)+Re^{-1}(\nabla \mathcal{F}[\bu^{n+1}],\nabla \bv_h)}\nonumber\\&&\nonumber
+(\mathcal{F}[\bu^{n+1}]\cdot\nabla \mathcal{F}[\bu^{n+1}],v_h) -s(\mathcal{F}[\bB^{n+1}]\cdot\nabla(\mathcal{F}[\bB^{n+1}]),\bv_h)\nonumber
\\
& =& (\bff(t^{n+1}),\bv_h)+E_1(\bu,\bB,\bv_h), \label{contu}
\end{eqnarray}
for all $v_h\in V_h$ and for the magnetic field
\begin{eqnarray}
\lefteqn{\left(\frac{\frac{3}{2} {\bB}^{n+1}-2\bB^n+\frac{1}{2}\bB^{n-1}}{\Delta t},\chi_h\right)+Re_m^{-1}(\nabla \mathcal{F}[\bB^{n+1}],\nabla \chi_h)}\nonumber
\\&&-(\mathcal{F}[\bB^{n+1}]\cdot\nabla \mathcal{F}[\bu^{n+1}],\chi_h)+(\mathcal{F}[\bu^{n+1}]\cdot\nabla \mathcal{F}[\bB^{n+1}],\chi_h) \nonumber
\\
&=&(\nabla \times
g(t^{n+1}),\chi_h)+E_2(\bu,\bB,\chi_{h}),  \label{contB}
\end{eqnarray}
for all $\chi_h\in V_h$ where \begin{eqnarray}
E_1(\bu,\bB,\bv_h)&=&\left(\frac{\frac{3}{2} {\bu}^{n+1}-2\bu^n+\frac{1}{2}\bu^{n-1}}{\Delta t}-u_t^{n+1},v_h\right)\nonumber\\
&&+Re^{-1}(\nabla \mathcal{F}[\bu^{n+1}],\nabla v_h)-Re^{-1}(\nabla u^{n+1},\nabla v_h)\nonumber\\
&&+(\mathcal{F}[\bu^{n+1}]\cdot\nabla \mathcal{F}[\bu^{n+1}],v_h)-(u^{n+1}\cdot\nabla u^{n+1},v_h)\nonumber\\
&&+s(B^{n+1}\cdot\nabla B^{n+1},v_h)-s(\mathcal{F}[\bB^{n+1}]\cdot\nabla \mathcal{F}[\bB^{n+1}],v_h),\\[0.3cm]
E_2(\bu,\bB,\chi_{h})&=&\left(\frac{\frac{3}{2} {\bB}^{n+1}-2\bB^n+\frac{1}{2}\bB^{n-1}}{\Delta t}-B_t^{n+1},\chi_h\right)\nonumber\\
&&+Re_m^{-1}(\nabla \mathcal{F}[\bB^{n+1}],\nabla \chi_h)-Re_m^{-1}(\nabla \bB^{n+1},\nabla \chi_h)\nonumber\\&&+(\bB^{n+1}\cdot\nabla \bu^{n+1},\chi_h)-(\mathcal{F}[\bB^{n+1}]\cdot\nabla \mathcal{F}[\bu^{n+1}],\chi_h)\nonumber\\&&+(\mathcal{F}[\bu^{n+1}]\cdot\nabla( \mathcal{F}[\bB^{n+1}]),\chi_h)-(u^{n+1}\cdot\nabla B^{n+1},\chi_h).
\end{eqnarray}
 Denote the error between finite
element solution and continuous solution by
$e_u^n:={u}^n-u_h^n$ and $e^n_B:={B}^n-B_h^n$. The error
equations are obtained by subtracting \eqref{sBE1}-\eqref{sBE3} from \eqref{contu}-\eqref{contB}, respectively:
\begin{eqnarray}
\lefteqn{\frac{1}{2\Delta t}(3e_{\bu}^{n+1}-4e_{\bu}^{n}+e_{\bu}^{n-1},v_{h})
+Re^{-1}(\nabla \mathcal{F}[e_{\bu}^{n+1}],\nabla v_{h}) + (\mathcal{F}[u_{h}^{n+1}]
\cdot \nabla \mathcal{F}[e_{\bu}^{n+1}],v_{h})} \nonumber
\\
&&+ (\mathcal{F}[e_{\bu}^{n+1}]\cdot \nabla \mathcal{F}[ u^{n+1}],v_{h})-s(\mathcal{F}[B_h^{n+1}] \cdot \nabla \mathcal{F}[e_{\bB}^{n+1}],v_{h}) -s(\mathcal{F}[e_{\bB}^{n+1}] \cdot
\nabla \mathcal{F}[B^{n+1}],v_{h})\nonumber\\ &=& E_1(\bu,\bB,\bv_h), \label{erru}
\end{eqnarray}
and
\begin{eqnarray}
\lefteqn{\frac{1}{2\Delta
t}(3e_{B}^{n+1}-4e_{B}^{n}+e_{B}^{n-1},\chi_{h}) +Re_{m}^{-1}(\nabla
\mathcal{F}[e_{\bB}^{n+1}],\nabla\chi_{h})
-(\mathcal{F}[B_h^{n+1}] \cdot \nabla \mathcal{F}[e_{u}^{n+1}],\chi_{h})}
 \nonumber
\\
&&-(\mathcal{F}[e_{\bB}^{n+1}]\cdot \nabla \mathcal{F}[u^{n+1}],\chi_{h})+(\mathcal{F}[u_{h}^{n+1}]\cdot \nabla \mathcal{F}[e_{\bB}^{n+1}],\chi_h)
+(\mathcal{F}[e_{\bu}^{n+1}]\cdot \nabla \mathcal{F}[B^{n+1}],\chi_{h}) \nonumber\\& =& E_2(\bu,\bB,\chi_h).\label{errB}
\end{eqnarray}
We split the errors as follows
\begin{eqnarray}
e_{\bu}^{n}&=& \bu^n-{\bu_h}^n=(\bu^n-I_{\bu}^n)-({\bu_h}^n-I_{\bu}^n)=\eta_{\bu}^n-\phi_{\bu,h}^n, \label{dec1}
\\
e_{\bB}^{n}&=& \bB^n-{\bB_h}^n=(\bB^n-I_{\bB}^n)-({\bB_h}^n-I_{\bB}^n)=\eta_{\bB}^n-\phi_{\bB,h}^n, \label{dec2}
\end{eqnarray}
where $I_{\bu}^n$ and $I_{\bB}^n$ are the interpolations of
$u^n$ and $B^n$ in $V_h$, respectively. Substituting \eqref{dec1} into \eqref{erru} and \eqref{dec2} into \eqref{errB}, choosing $v_h=\mathcal{F}[\phi_{\bu,h}^{n+1}]$ and using Lemma \ref{lem:inn} and \eqref{divuB}, 
leads to
\begin{eqnarray}
\lefteqn{\frac{1}{\Delta t}\norm{ \begin{bmatrix}
		\phi_{\bu,h}^{n+1} \\
	\phi_{\bu,h}^{n}
		\end{bmatrix}} _{G}^{2} -\frac{1}{\Delta t}\norm { \begin{bmatrix}
		\phi_{\bu,h}^{n} \\
		\phi_{\bu,h}^{n-1}
		\end{bmatrix}}_{G}^{2} 	+\dfrac{1}{4\Delta t}\norm{\phi_{\bu,h}^{n+1}-2\phi_{\bu,h}^{n}+\phi_{\bu,h}^{n-1}}_{F}^{2}}\nonumber\\&&
+Re^{-1}\|{\nabla \mathcal{F}[\phi_{\bu,h}^{n+1}]}\|^{2} \nonumber
\\
&\leq&|(\frac{3\eta_u^{n+1}-4\eta_u^n+\eta_u^{n-1}}{2\Delta t},\mathcal{F}[\phi_{\bu,h}^{n+1}])|+Re^{-1}| (\nabla
\mathcal{F}[\eta_u^{n+1}],\nabla  \mathcal{F}[\phi_{\bu,h}^{n+1}])|\nonumber\\&&+s(\mathcal{F}[\bB_h^{n+1}]\cdot \nabla
\mathcal{F}[\phi_{\bB,h}^{n+1}],\mathcal{F}[\phi_{\bu,h}^{n+1}])+|(\mathcal{F}[\bu_h^{n+1}]\cdot \nabla \mathcal{F}[\eta_u^{n+1}],\mathcal{F}[\phi_{\bu,h}^{n+1}])| \nonumber
\\[0.1cm]
&& + |(\mathcal{F}[\phi_{\bu,h}^{n+1}]\cdot \nabla \mathcal{F}[ u^{n+1}],\mathcal{F}[\phi_{\bu,h}^{n+1}])|+|(\mathcal{F}[\eta_u^{n+1}]\cdot \nabla \mathcal{F}[u^{n+1}],\mathcal{F}[\phi_{\bu,h}^{n+1}])|
  \nonumber
\\[0.1cm]
&&+s|(\mathcal{F}[\bB_h^{n+1}]\cdot \nabla
\mathcal{F}[\eta_B^{n+1}],\mathcal{F}[\phi_{\bu,h}^{n+1}])+s|(\mathcal{F}[\phi_{\bB,h}^{n+1}]\cdot \nabla \mathcal{F}[B^{n+1}],\mathcal{F}[\phi_{\bu,h}^{n+1}])| \nonumber
\\[0.1cm]
&&+s|(\mathcal{F}[\eta_B^{n+1}]\cdot \nabla \mathcal{F}[B^{n+1}],\mathcal{F}[\phi_{\bu,h}^{n+1}])| +|E_1(\bu,\bB,\mathcal{F}[\phi_{\bu,h}^{n+1}])|. \label{velerror1}
\end{eqnarray}

Then, we now bound the terms in the right hand side of (\ref{velerror1}) and obtain
\begin{align}
|(\frac{3\eta_u^{n+1}-4\eta_u^n+\eta_u^{n-1}}{2\Delta t},\mathcal{F}[\phi_{\bu,h}^{n+1}])|
	&\leq
	CRe\|\frac{3\eta_u^{n+1}-4\eta_u^n+\eta_u^{n-1}}{2\Delta t}\|^2
	+\frac{Re^{-1}}{24}\| \nabla \mathcal{F}[\phi_{\bu,h}^{n+1}]\|^{2},\label{boundvel1}\\
Re^{-1}| (\nabla
\mathcal{F}[\eta_u^{n+1}],\nabla  \mathcal{F}[\phi_{\bu,h}^{n+1}])|&\leq CRe^{-1}\|\nabla
\mathcal{F}[\eta_u^{n+1}]\|^2+\frac{Re^{-1}}{24}\|\nabla \mathcal{F}[\phi_{\bu,h}^{n+1}]\|^2,
\end{align}
for the first two terms along with the Cauchy-Schwarz and Young's inequalities. Also, with Lemma {\ref{tribound}}, Cauchy-Schwarz and Young's inequalities, we get estimations for the nonlinear terms:
\begin{align}
|(\mathcal{F}[\bu_h^{n+1}]\cdot \nabla \mathcal{F}[\eta_u^{n+1}],\mathcal{F}[\phi_{\bu,h}^{n+1}])| \leq& CRe\|\nabla\mathcal{F}[ \bu_h^{n+1}]\|^2\|\nabla \mathcal{F}[\eta_u^{n+1}]\|^2+\frac{Re^{-1}}{24}\|\nabla \mathcal{F}[\phi_{\bu,h}^{n+1}]\|^2,\allowdisplaybreaks\\\allowdisplaybreaks
|(\mathcal{F}[\phi_{\bu,h}^{n+1}]\cdot \nabla \mathcal{F}[u^{n+1}],\mathcal{F}[\phi_{\bu,h}^{n+1}])|\leq& C{Re}^3\|\mathcal{F}[\phi_{\bu,h}^{n+1}]\|^2\| \nabla \mathcal{F}[u^{n+1}]\|^4+\frac{{Re}^{-1}}{24}\|\nabla \mathcal{F}[\phi_{\bu,h}^{n+1}]\|^2,
\\|(\mathcal{F}[\eta_u^{n+1}]\cdot \nabla \mathcal{F}[u^{n+1}],\mathcal{F}[\phi_{\bu,h}^{n+1}])|\leq& CRe\|\nabla \mathcal{F}[\eta_u^{n+1}]\|^2\| \nabla \mathcal{F}[u^{n+1}]\|^2+\frac{Re^{-1}}{24}\|\nabla \mathcal{F}[\phi_{\bu,h}^{n+1}]\|^2,\allowdisplaybreaks\\
s|(\mathcal{F}[\bB_h^{n+1}]\cdot \nabla
\mathcal{F}[\eta_B^{n+1}],\mathcal{F}[\phi_{\bu,h}^{n+1}])\leq& Cs^2Re\|\nabla \mathcal{F}[\bB_h^{n+1}]\|^2\|\nabla
\mathcal{F}[\eta_B^{n+1}]\|^2+\frac{Re^{-1}}{24}\|\nabla \mathcal{F}[\phi_{\bu,h}^{n+1}]\|^2,\allowdisplaybreaks\\
s|(\mathcal{F}[\phi_{\bB,h}^{n+1}]\cdot \nabla \mathcal{F}[\bB^{n+1}],\mathcal{F}[\phi_{\bu,h}^{n+1}])|\leq& Cs^4{Re}^{2}{Re}_m\|\mathcal{F}[\phi_{\bB,h}^{n+1}]\|^2\| \nabla \mathcal{F}[\bB^{n+1}]\|^4\nonumber\\
&+\frac{Re_m^{-1}}{4}\|\nabla \mathcal{F}[\phi_{\bB,h}^{n+1}]\|^2+\frac{Re^{-1}}{24}\|\nabla \phi_{\bu,h}\|^2,
\allowdisplaybreaks\\s|(\mathcal{F}[\eta_B^{n+1}]\cdot \nabla \mathcal{F}[\bB^{n+1}],\mathcal{F}[\phi_{\bu,h}^{n+1}])| \leq& Cs^2Re\|\nabla\mathcal{F}[\eta_B^{n+1}]\|^2\| \nabla \mathcal{F}[\bB^{n+1}]\|^2+\frac{Re^{-1}}{24}\|\nabla \mathcal{F}[\phi_{\bu,h}^{n+1}]\|^2.\label{boundvel2}
\end{align}
In addition, the terms in consistency error $|E_1(\bu,\bB,\mathcal{F}[\phi_{\bu,h}^{n+1}])|$ are bounded by using Cauchy-Schwarz, Poincar\`{e} and Young's inequalities as follows:
\begin{eqnarray}
\lefteqn{\abs{\left(\frac{{3} {\bu}^{n+1}-4\bu^n+\bu^{n-1}}{2\Delta t}-\bu_t^{n+1},\mathcal{F}[\phi_{\bu,h}^{n+1}]\right)}}\nonumber\\&\leq& CRe\|\frac{{3} {\bu}^{n+1}-4\bu^n+\bu^{n-1}}{2\Delta t}-\bu_t^{n+1}\|^2+\frac{Re^{-1}}{24}\|\nabla \mathcal{F}[\phi_{\bu,h}^{n+1}]\|^2,\\[0.4cm]
\lefteqn{Re^{-1}|(\nabla \mathcal{F}[\bu^{n+1}]-u^{n+1},\nabla \mathcal{F}[\phi_{\bu,h}^{n+1}])|}\nonumber\\&\leq&CRe^{-1}\|\nabla \mathcal{F}[\bu^{n+1}]-u^{n+1}\|^2+\frac{Re^{-1}}{24}\|\nabla \mathcal{F}[\phi_{\bu,h}^{n+1}]\|^2,
\\[0.4cm]
\lefteqn{(\mathcal{F}[\bu^{n+1}]\cdot\nabla \mathcal{F}[\bu^{n+1}],\mathcal{F}[\phi_{\bu,h}^{n+1}])-(u^{n+1}\cdot\nabla u^{n+1},\mathcal{F}[\phi_{\bu,h}^{n+1}])}\nonumber
\\&\leq& CRe(\|\nabla \mathcal{F}[\bu^{n+1}]\|^2 +\|\nabla {u}^{n+1}\|^2)\|\nabla \mathcal{F}[\bu^{n+1}]-u^{n+1}\|^2\nonumber\\&&+\frac{Re^{-1}}{24}\|\nabla \mathcal{F}[\phi_{\bu,h}^{n+1}]\|^2,\\
[0.4cm]
\lefteqn{s(B^{n+1}\cdot\nabla B^{n+1},\mathcal{F}[\phi_{\bu,h}^{n+1}])-s(\mathcal{F}[\bB^{n+1}]\cdot\nabla \mathcal{F}[\bB^{n+1}],\mathcal{F}[\phi_{\bu,h}^{n+1}])}\nonumber
\\&\leq& CRe(\|\nabla \mathcal{F}[\bB^{n+1}]\|^2 +\|\nabla {B}^{n+1}\|^2)\|\nabla \mathcal{F}[\bB^{n+1}]-\bB^{n+1}\|^2\nonumber\\&&+\frac{Re^{-1}}{24}\|\nabla \mathcal{F}[\phi_{\bu,h}^{n+1}]\|^2.\label{sbE}
\end{eqnarray}

Inserting  (\ref{boundvel1})-(\ref{sbE}) into \eqref{velerror1} yields
\begin{eqnarray}
\lefteqn{\frac{1}{\Delta t}\norm{ \begin{bmatrix}
		\phi_{\bu,h}^{n+1} \\
		\phi_{\bu,h}^{n}
		\end{bmatrix}} _{G}^{2} -\frac{1}{\Delta t}\norm { \begin{bmatrix}
		\phi_{\bu,h}^{n} \\
		\phi_{\bu,h}^{n-1}
		\end{bmatrix}}_{G}^{2} 	+\dfrac{1}{4\Delta t}\norm{\phi_{\bu,h}^{n+1}-2\phi_{\bu,h}^{n}+\phi_{\bu,h}^{n-1}}_{F}^{2}}\nonumber\\
	&&
 +\frac{Re^{-1}}{2}\|{\nabla \mathcal{F}[\phi_{\bu,h}^{n+1}]}\|^{2} \nonumber
\\
&\leq& CRe\bigg(
\|\frac{3\eta_u^{n+1}-4\eta_u^n+\eta_u^{n-1}}{2\Delta t}\|^2
+Re^{-2}\|\nabla\mathcal{F}[\eta_u^{n+1}]\|^2+\|\nabla\mathcal{F}[ u^{n+1}]\|^2 \|\nabla \mathcal{F}[\eta_u^{n+1}]\|^2
\nonumber \\
&&
+ Re^2\|\mathcal{F}[\phi_{\bu,h}^{n+1}]\|^2 \|\nabla \mathcal{F} [\bu^{n+1}]\|^4
+\|\nabla \mathcal{F}[\eta_u^{n+1}]\|^2 \|\nabla \mathcal{F}[u^{n+1}]\|^2
\nonumber \\[0.1cm]
&&
+s^2\|\nabla \mathcal{F}[\bB_h^{n+1}]\| \|  \nabla\mathcal{F}[\eta_B^{n+1}]\|
+s^4{Re}{Re}_m\|\mathcal{F}[\phi_{\bB,h}^{n+1}]\|^2\| \nabla \mathcal{F}[B^{n+1}]\|^4\nonumber \\
&& +s^2\|\nabla\mathcal{F}[\eta_B^{n+1}]\|^2\| \nabla \mathcal{F}[B^{n+1}]\|^2
+\|\frac{{3} {\bu}^{n+1}-4\bu^n+\bu^{n-1}}{2\Delta t}-\bu_t^{n+1}\|^2\nonumber\\
&&+\Big(Re^{-2}+\|\nabla \mathcal{F}[\bu^{n+1}]\|^2 +\|\nabla {u}^{n+1}\|^2\Big)\|\nabla \mathcal{F}[\bu^{n+1}]-u^{n+1}\|^2\nonumber\\&&+\Big(\|\nabla \mathcal{F}[\bB^{n+1}]\|^2 +\|\nabla {B}^{n+1}\|^2\Big)\|\nabla \mathcal{F}[\bB^{n+1}]-\bB^{n+1}\|^2\bigg)\nonumber\\&&+s(\mathcal{F}[\bB_h^{n+1}]\cdot \nabla\mathcal{F}[\phi_{\bB,h}^{n+1}],\mathcal{F}[\phi_{\bu,h}^{n+1}])+\frac{Re_m^{-1}}{4}\|\nabla \mathcal{F}[\phi_{\bB,h}^{n+1}]\|^2.\label{velerror21}
\end{eqnarray}
Multiplying (\ref{velerror21}) by $\Delta t$  and summing from $t=1$ to $t=N-1$, we have
\begin{eqnarray}
\lefteqn{\norm{ \begin{bmatrix}
		\phi_{\bu,h}^{N} \\
		\phi_{\bu,h}^{N-1}
		\end{bmatrix}} _{G}^{2} 	+\dfrac{1}{4}\sum_{n=1}^{N-1}\norm{\phi_{\bu,h}^{n+1}-2\phi_{\bu,h}^{n}+\phi_{\bu,h}^{n-1}}_{F}^{2}\nonumber
+\frac{Re^{-1}\Delta t}{2}\sum_{n=1}^{N-1}\|{\nabla \mathcal{F}[\phi_{\bu,h}^{n+1}]}\|^{2} }\nonumber
\allowdisplaybreaks\\\allowdisplaybreaks
&\leq& \norm { \begin{bmatrix}
	\phi_{\bu,h}^{1} \\
	\phi_{\bu,h}^{0}
	\end{bmatrix}}_{G}^{2} +C\bigg(
\Delta t\sum_{n=1}^{N-1}\|\frac{3\eta_u^{n+1}-4\eta_u^n+\eta_u^{n-1}}{2\Delta t}\|^2
+\Delta t\sum_{n=1}^{N-1}\|\nabla\mathcal{F}[\eta_u^{n+1}]\|^2\nonumber\allowdisplaybreaks\\&&+\Delta t\sum_{n=1}^{N-1}\|\nabla\mathcal{F}[ u^{n+1}]\|^2 \|\nabla \mathcal{F}[\eta_u^{n+1}]\|^2+\Delta t\sum_{n=1}^{N-1} Re^3\|\mathcal{F}[\phi_{\bu,h}^{n+1}]\|^2 \|\nabla \mathcal{F}[u^{n+1}]\|^4
\nonumber \allowdisplaybreaks\\
&&
+\Delta t\sum_{n=1}^{N-1}\|\nabla \mathcal{F}[\eta_u^{n+1}]\|^2 \|\nabla \mathcal{F}[u^{n+1}]\|^2 +s^2\Delta t\sum_{n=1}^{N-1}\|\nabla \mathcal{F}[\bB_h^{n+1}]\| \|  \nabla\mathcal{F}[\eta_B^{n+1}]\|
\nonumber\allowdisplaybreaks \\
&&
+s^4{Re}^2{Re}_m\Delta t\sum_{n=1}^{N-1}\|\mathcal{F}[\phi_{\bB,h}^{n+1}]\|^2\| \nabla \mathcal{F}[B^{n+1}]\|^4\nonumber \allowdisplaybreaks\\
&& +s^2\Delta t\sum_{n=1}^{N-1}\|\nabla\mathcal{F}[\eta_B^{n+1}]\|^2\| \nabla \mathcal{F}[B^{n+1}]\|^2
+\Delta t\sum_{n=1}^{N-1}\|\frac{{3} {\bu}^{n+1}-4\bu^n+\bu^{n-1}}{2\Delta t}-\bu_t^{n+1}\|^2\nonumber\\
&&+\Delta t\sum_{n=1}^{N-1}\Big(Re^{-2}+\|\nabla \mathcal{F}[\bu^{n+1}]\|^2 +\|\nabla {u}^{n+1}\|^2\Big)\|\nabla \mathcal{F}[\bu^{n+1}]-u^{n+1}\|^2\nonumber\\&&+\Delta t\sum_{n=1}^{N-1}\Big(\|\nabla \mathcal{F}[\bB^{n+1}]\|^2 +\|\nabla {B}^{n+1}\|^2\Big)\|\nabla \mathcal{F}[\bB^{n+1}]-\bB^{n+1}\|^2\bigg)\nonumber\\&&+s\Delta t\sum_{n=1}^{N-1}(\mathcal{F}[\bB_h^{n+1}]\cdot \nabla\mathcal{F}[\phi_{\bB,h}^{n+1}],\mathcal{F}[\phi_{\bu,h}^{n+1}])+\frac{Re_m^{-1}\Delta t}{4}\sum_{n=1}^{N-1}\|\nabla \mathcal{F}[\phi_{\bB,h}^{n+1}]\|^2.\label{velerror3}
\end{eqnarray}
Using Lemma \ref{lem:fw} and approximation properties \eqref{app1}-\eqref{app2}, we have
\begin{eqnarray}
\Delta t \sum_{n=1}^{N-1}\|\frac{3\eta_u^{n+1}-4\eta_u^n+\eta_u^{n-1}}{2\Delta t}\|^2
&\leq&{C }h^{2s+2}||\bu_{t}||^2_{L^2(0,T;H^{s+1}(\Omega))},\label{err1}\\
\Delta t \sum_{n=1}^{N-1}\|\nabla(\mathcal{F}[\eta_u^{n+1}])\|^2&\leq&Ch^{2s} |||\bu|||^2 _{2,s+1},\\
\Delta t \sum_{n=1}^{N-1}\|\nabla(\mathcal{F}[\eta_B^{n+1}])\|^2&\leq&Ch^{2s} |||\bB|||^2 _{2,s+1},\\
\Delta t \sum_{n=1}^{N-1}\|\frac{{3} {\bu}^{n+1}-4\bu^n+\bu^{n-1}}{2\Delta t}-\bu_t^{n+1}\|^2&\leq& C\Delta t^4\|\bu_{ttt}\|_{L^2(0,T;L^2(\Omega))}^2,\\
\Delta t \sum_{n=1}^{N-1}\|\nabla \mathcal{F}[\bu^{n+1}]-u^{n+1}\|^2&\leq& C\Delta t^4 \|\nabla \bu_{tt}\|^2_{L^2(0,T;L^2(\Omega))},\allowdisplaybreaks\\\allowdisplaybreaks
\Delta t \sum_{n=1}^{N-1}\|\nabla \mathcal{F}[\bB^{n+1}]-\bB^{n+1}\|^2&\leq& C\Delta t^4 \|\nabla \bB_{tt}\|^2_{L^2(0,T;L^2(\Omega))}.\label{err6}
\end{eqnarray}
Substituting \eqref{err1}-\eqref{err6} in (\ref{velerror3}) and utilizing Lemma \ref{l:stab}, one gets
\begin{eqnarray}
\lefteqn{\norm{ \begin{bmatrix}
		\phi_{\bu,h}^{n+1} \\
		\phi_{\bu,h}^{n}
		\end{bmatrix}} _{G}^{2} -\norm { \begin{bmatrix}
		\phi_{\bu,h}^{n} \\
		\phi_{\bu,h}^{n-1}
		\end{bmatrix}}_{G}^{2} 	+\dfrac{1}{4}\norm{\phi_{\bu,h}^{n+1}-2\phi_{\bu,h}^{n}+\phi_{\bu,h}^{n-1}}_{F}^{2}}\nonumber\\
&&
+\frac{Re^{-1}\Delta t}{2}\|{\nabla \mathcal{F}[\phi_{\bu,h}^{n+1}]}\|^{2} \nonumber
\\
&\leq& C\bigg({h^{2s+2}}||\bu_{t}||^2_{{L^2(0,T;H^{s+1}(\Omega))}}
+(Re^{-2}+|||\nabla u|||_{\infty,0}^2 )h^{2s} |||\bu|||^2_{2,s+1}\nonumber\\
&&+\Big((\frac{1}{3})^N(\|{\bu}_h^{0}\|^2+s\|{\bB}_h^{0}\|^2)+N( 2(\|{\bu}_h^{1}\|^2+s\|{\bB}_h^{1}\|^2)+(\|{\bu}_h^{0}\|^2+s\|{\bB}_h^{0}\|^2) ) \nonumber\allowdisplaybreaks\\
&&+\frac{2 N\Delta t}{3}\sum_{n=1}^{N-1}(Re\|\bff(t^{n+1})\|_{-1}^2+s Re_m\|\bfg(t^{n+1})\|^2)\Big)h^{2s} (|||{\bu}|||_{2,s+1}^2+|||{\bB}|||_{2,s+1}^2)\nonumber\allowdisplaybreaks\\
&& +Re^3|||\nabla u|||_{\infty,0}^4 \Delta t\sum_{n=1}^{N-1}\|\mathcal{F}[\phi_{\bu,h}^{n+1}]\|^2
+s^4{Re}^2{Re}_m||| \nabla \mathcal{F}[B]|||_{\infty,0}^4\sum_{n=1}^{N-1}\|\mathcal{F}[\phi_{\bB,h}^{n+1}]\|^2\nonumber\allowdisplaybreaks\\&&+(s||| \nabla\mathcal{F}[ B]|||_{\infty,0}^2)h^{2s} |||\bB|||^2_{2,s+1}+\Delta t^4||\bu_{ttt}||_{L^2(0,T;L^2(\Omega))}^2\nonumber\allowdisplaybreaks \\
&&
+\Big(Re^{-2}+|||\nabla \mathcal{F}[\bu]|||_{\infty,0}^2 +|||\nabla {u}|||_{\infty,0}^2\Big)\Delta t^4 ||\nabla \bu_{tt}||^2_{L^2(0,T;L^2(\Omega))}\nonumber\\&&+\Big(|||\nabla \mathcal{F}[\bB]|||_{\infty,0}^2 +|||\nabla {B}|||_{\infty,0}^2\Big)\Delta t^4 ||\nabla \bB_{tt}||^2_{L^2(0,T;L^2(\Omega))}\bigg)\nonumber\\&&+s\Delta t\sum_{n=1}^{N-1}(\mathcal{F}[\bB_h^{n+1}]\cdot \nabla\mathcal{F}[\phi_{\bB,h}^{n+1}],\mathcal{F}[\phi_{\bu,h}^{n+1}])+\frac{Re_m^{-1}\Delta t}{4}\sum_{n=1}^{N-1}\|\nabla \mathcal{F}[\phi_{\bB,h}^{n+1}]\|^2. \label{velerror5}
\end{eqnarray}
Reorganizing equation (\ref{velerror5}), we have
\begin{eqnarray}
\lefteqn{\norm{ \begin{bmatrix}
		\phi_{\bu,h}^{N} \\[5pt]
		\phi_{\bu,h}^{N-1}
		\end{bmatrix}} _{G}^{2}
	+\dfrac{1}{4}\sum_{n=1}^{N-1}\norm{\phi_{\bu,h}^{n+1}-2\phi_{\bu,h}^{n}+\phi_{\bu,h}^{n-1}}_{F}^{2}+\frac{Re^{-1}\Delta t}{2}\sum_{n=1}^{N-1}\|{\nabla( \mathcal{F}[\phi_{\bu,h}^{n+1}])}\|^{2}} \nonumber
\\
&\leq& \norm{ \begin{bmatrix}
	\phi_{\bu,h}^{1} \\[5pt]
	\phi_{\bu,h}^{0}
	\end{bmatrix}} _{G}^{2}+C(
h^{2s+2}+ h^{2s} +\Delta t^4)+Re^3 |||\nabla \mathcal{F}[\bu]|||_{\infty,0}^4\Delta t\sum_{n=1}^{N-1}\|\mathcal{F}[\phi_{\bu,h}^{n+1}]\|^2
\nonumber \\
&&
+s^4{Re}^2{Re}_m||| \nabla \mathcal{F}[B]|||_{\infty,0}^4\Delta t\sum_{n=1}^{N-1}\|\mathcal{F}[\phi_{\bB,h}^{n+1}]\|^2+\sum_{n=1}^{N-1}\frac{Re_m^{-1}}{4}\|\nabla \mathcal{F}[\phi_{\bB,h}^{n+1}]\|^2\nonumber\\
&& +s\Delta t\sum_{n=1}^{N-1}(\mathcal{F}[\bB_h^{n+1}]\cdot \nabla\mathcal{F}[\phi_{\bB,h}^{n+1}],\mathcal{F}[\phi_{\bu,h}^{n+1}]). \label{velerror4}
\end{eqnarray}
In a similar manner, substituting \eqref{dec2} into \eqref{errB}, and setting $\chi_h = \mathcal{F}[\phi_{\bB,h}^{n+1}]$ gives
\begin{eqnarray}
\lefteqn{\norm{ \begin{bmatrix}
		\phi_{\bB,h}^{N} \\[5pt]
		\phi_{\bB,h}^{N-1}
		\end{bmatrix}} _{G}^{2}
	+\dfrac{1}{4}\sum_{n=1}^{N-1}\norm{\phi_{\bB,h}^{n+1}-2\phi_{\bB,h}^{n}+\phi_{\bB,h}^{n-1}}_{F}^{2}
	+\frac{Re_m^{-1}\Delta t}{2}\sum_{n=1}^{N-1}\|{\nabla \mathcal{F}[\phi_{\bB,h}^{n+1}]}\|^{2}} \nonumber
\\
&\leq& \norm{ \begin{bmatrix}
	\phi_{\bB,h}^{1} \\[5pt]
	\phi_{\bB,h}^{0}
	\end{bmatrix}} _{G}^{2}+C \Big(
h^{2s+2}+ h^{2s}+\Delta t^4\Big)
+sReRe_m^2 \| \nabla \mathcal{F}[B]\|_{\infty,0}^4\Delta t\sum_{n=1}^{N-1}\|\mathcal{F}[\phi_{\bu,h}^{n+1}]\|^2\nonumber\\
&& +Re_m^3 \|\nabla \mathcal{F}[\bu]\|_{\infty,0}^4\Delta t\sum_{n=1}^{N-1}\|\mathcal{F}[\phi_{\bB,h}^{n+1}]\|^2+\sum_{n=1}^{N-1}\frac{Re^{-1}\Delta t}{4s}\|\nabla \mathcal{F}[\phi_{\bu,h}^{n+1}]\|^2\nonumber\\&&+\Delta t\sum_{n=1}^{N-1}(\mathcal{F}[\bB_h^{n+1}]\cdot \nabla\mathcal{F}[\phi_{\bu,h}^{n+1}],\mathcal{F}[\phi_{\bB,h}^{n+1}]). \label{magerror4}
\end{eqnarray}
Multiplying (\ref{magerror4}) by $s $, adding it to (\ref{velerror4}) and using that $$(\mathcal{F}[{B_h}^{n+1}]\cdot \nabla \mathcal{F}[\phi_{\bu,h}^{n+1}],\mathcal{F}[\phi_{\bB,h}^{n+1}])=-(\mathcal{F}[{B_h}^{n+1}]\cdot
\nabla\mathcal{F}[\phi_{\bB,h}^{n+1}],\mathcal{F}[\phi_{\bu,h}^{n+1}]),$$ we get
\begin{eqnarray}
\lefteqn{\norm{ \begin{bmatrix}
		\phi_{\bu,h}^{N} \\[5pt]
		\phi_{\bu,h}^{N-1}
		\end{bmatrix}} _{G}^{2}+s\norm{ \begin{bmatrix}
		\phi_{\bB,h}^{N} \\[5pt]
		\phi_{\bB,h}^{N-1}
		\end{bmatrix}} _{G}^{2}
	+\frac{\Delta t}{4}\sum_{n=0}^{N-1}(Re^{-1}\|{\nabla \mathcal{F}[\phi_{\bu,h}^{n+1}]}\|^{2}
	+sRe_m^{-1}\|{\nabla \mathcal{F}[\phi_{\bB,h}^{n+1}]}\|^{2})} \nonumber
\allowdisplaybreaks\\
&&+\dfrac{1}{4}\sum_{n=0}^{N-1}(\norm{\phi_{\bu,h}^{n+1}-2\phi_{\bu,h}^{n}+\phi_{\bu,h}^{n-1}}_{F}^{2}+s\norm{\phi_{\bB,h}^{n+1}-2\phi_{\bB,h}^{n}+\phi_{\bB,h}^{n-1}}_{F}^{2})\nonumber\allowdisplaybreaks\\
&\leq& \norm{ \begin{bmatrix}
	\phi_{\bu,h}^{1} \\[5pt]
	\phi_{\bu,h}^{0}
	\end{bmatrix}} _{G}^{2}+ s\norm{ \begin{bmatrix}
	\phi_{\bB,h}^{1} \\[5pt]
	\phi_{\bB,h}^{0}
	\end{bmatrix}} _{G}^{2}+C(
h^{2s+2}+h^{2s}+{\Delta t}^4)
\nonumber\allowdisplaybreaks \\
&& +\Big( Re^3|||\nabla \mathcal{F}[\bu]|||_{\infty,0}^4+s^2ReRe_m^2 ||| \nabla \mathcal{F}[ B]|||_{\infty,0}^4\Big)\Delta t\sum_{n=0}^{N-1}\|\mathcal{F}[\phi_{\bu,h}^{n+1}]\|^2\nonumber\allowdisplaybreaks\\
&& +\Big(sRe_m^3 |||\nabla \mathcal{F}[\bu]|||_{\infty,0}^4+s^4{Re}^2{Re}_m||| \nabla \mathcal{F}[B]|||_{\infty,0}^4\Big)\Delta t\sum_{n=0}^{N-1}\|\mathcal{F}[\phi_{\bB,h}^{n+1}]\|^2
. \label{velmagerror1}
\end{eqnarray}
Application of the discrete Gronwall inequality with
\begin{eqnarray}
	\Delta t &\leq&C(s)\Big( |||\nabla\bu|||_{\infty,0}^4+ ||| \nabla  B|||_{\infty,0}^4\Big)^{-1},
\end{eqnarray}
and utilization of Lemma \ref{lem:gnorm} yields
\begin{eqnarray}
\lefteqn{\dfrac{3}{4}(\|{\phi_{\bu,h}^{N}}\|^{2}+s\|{\phi_{\bB,h}^{N}}\|^{2})+\frac{\Delta t}{4}\sum_{n=0}^{N-1}(Re^{-1}\|{\nabla \mathcal{F}[\phi_{\bu,h}^{n+1}]}\|^{2}
	+sRe_m^{-1}\|{\nabla \mathcal{F}[\phi_{\bB,h}^{n+1}]}\|^{2})} \nonumber
\\
&\leq&\dfrac{1}{4}(\|{\phi_{\bu,h}^{N-1}}\|^{2}+s\|{\phi_{\bB,h}^{N-1}}\|^{2})+ \frac{3}{2}\Big(\norm{\phi_{\bu,h}^{1}}^{2}+s\norm{\phi_{\bB,h}^{1}}^{2}\nonumber\\&&+\norm{\phi_{\bu,h}^{0}}^{2}+s\norm{\phi_{\bB,h}^{0}}^{2}\Big)+C(
h^{2s}+{\Delta t}^4). \label{velmagerror2}
\end{eqnarray}
Multiplying (\ref{velmagerror2}) with $\frac{4}{3}$ and applying induction produces
\begin{eqnarray}
\lefteqn{\|{\phi_{\bu,h}^{N}}\|^{2}+s\|{\phi_{\bB,h}^{N}}\|^{2}+\frac{\Delta t}{3}\sum_{n=0}^{N-1}(Re^{-1}\|{\nabla \mathcal{F}[\phi_{\bu,h}^{n+1}]}\|^{2}
	+sRe_m^{-1}\|{\nabla \mathcal{F}[\phi_{\bB,h}^{n+1}]}\|^{2})} \nonumber
\\
&\leq&\Big(\frac{1}{3}\Big)^N(\|{\phi_{\bu,h}^{0}}\|^{2}+s\|{\phi_{\bB,h}^{0}}\|^{2})+ 2N\Big(\norm{\phi_{\bu,h}^{1}}^{2}+s\norm{\phi_{\bB,h}^{1}}^{2}\nonumber\\&&+\norm{\phi_{\bu,h}^{0}}^{2}+s\norm{\phi_{\bB,h}^{0}}^{2}\Big)+C(
h^{2s}+{\Delta t}^4). \label{velmagerror3}
\end{eqnarray}
The proof is completed by applying the triangle inequality.
\end{proof}
\section{Numerical Studies}
In this section, Algorithm \ref{algo} presented in Section 3 will be studied at examples given in a two-dimensional domain $\Omega$. We perform three different numerical tests in order to expose the promise of proposed method. The first example has been designed to confirm the theoretically predicted results of Theorem \ref{The:conv}. In the second test, we check the energy and cross-helicity conservation properties of the scheme for an ideal MHD case. In the final test, we investigate the flow behavior in a channel over a step under the effect of magnetic field.  The initial velocity, the initial pressure and the initial magnetic field were computed as nodal interpolants if not stated otherwise. For all simulations, the Scott-Vogelious pair of finite elements  $( (P_{2})^2, P_{1}^{disc})$ on  barycenter refined triangular meshes is used. The computations were performed with the public license finite element software FreeFem++ \cite{hec}.

\subsection{Convergence Rate Verification}

We consider the MHD equation (\ref{mhd5})-(\ref{mhd8}) in the unit square and in the time interval $[0,1]$ where the right hand side and the boundary conditions are chosen such that
\begin{align}\label{truesol1}
\bu=\left(%
\begin{array}{c}
y^5 + t^2 \\
x^5 + t^2%
\end{array}%
\right),\quad
p=10(2x-1)(2y-1)(1+t^2),\quad
B=\left(%
\begin{array}{c}
t^2 + \sin y \\
t^2 + \sin x %
\end{array}\right)
\end{align}
is the solution. Other problem parameters are chosen as $Re = Re_m = s = 1$. Since we are studying convergence, the spatial meshwidth $h$ and the time step $\Delta t$ are set to be same in order to see the errors and rates at once.  We measure the errors in the discrete norm  $L^2(0,T;{ H}^1(\Omega))$ for the velocity and the magnetic field which could be written for the velocity for example:
 $$\|{u}-{u}^h\|_{2,1}=\left\{\Delta t \sum_{n=1}^{N}\|{u}(t^n)-{u}_{n}^{h}\|^{2}\right\}^{1/2}.$$
Table \ref{RateTable} reports the order of convergence for Algorithm \ref{algo}. One can observe the predicted  second order convergence for the errors estimated in Theorem \ref{The:conv}.
\begin{table}[h!]
\begin{center}
\begin{tabular}{|c|c|c|c|c|c|c|}
\hline
$h=\Delta t$ & $\|u-u_{h}\|_{2,1}$& rate &  $ \| B-B_{h} \|_{2,1}$ & rate \\
\hline
1/2 & 0.30650 & - & 0.01460 & - \\
\hline
1/4 & 0.08936 & 1.73 & 0.00995 & 0.63 \\
\hline
1/8 & 0.02239 & 2.01 & 0.00282 & 1.81 \\
\hline
1/16 & 0.00559 & 2.01 & 0.00071 & 1.98 \\
\hline
1/32 & 0.00139 & 2.00 & 0.00017 & 2.00\\
\hline
1/64 & 0.00034 & 2.04 & 4.44e-5 & 2.01\\
\hline
\end{tabular}
\end{center}
\caption{Errors and rates of convergence for the velocity and the magnetic field.}
\label{RateTable}
\end{table}
We note that this test was also carried out for both with filtering and not filtering the pressure. More precisely, $\tilde{P}_{h}^{n+1}$ in (\ref{BE1}) and $\tilde{\lambda}_{h}^{n+1}$ in (\ref{BE3}) are chosen as $P_{h}^{n+1}$ and $\lambda_{h}^{n+1}$ and not updated in Step $2$ of Algorithm \ref{algo}. In both ways, we obtain the same error rates showing that the pressure filtering does not affect the velocity and magnetic field solution, exactly the same situation for Navier-Stokes equations, \cite{DLZ19}.

\subsection{Orszag-Tang Vortex Test}

As a second numerical test, we solve Orszag-Tang vortex problem which is a well-known model for testing MHD codes. Due to the complex interaction between various shock waves traveling at different speed regimes, this problem tests robustness of the code in the formation of shocks and shock-shock interactions in the ideal MHD case, (see \cite{friedel, LW0} and references therein). In addition, since the numerical solution of Orzag-Tang vortex system does not necessarily  preserve the incompressible constraint $\nabla \cdot B=0$, this problem also provides some quantitative estimations for the effect of significant magnetic monopoles on the numerical solutions. In this test problem, our goal is to show the confirmation of the conserved quantities and compare the results with unfiltered case in order to see the effect of the time filter explained in detail in Section \ref{cons}.
The problem is solved in $[0,2\pi] \times [0,2\pi] $ using the meshwidth $h=1/32$, the time step $\Delta t=0.01$ and the final time $2.7$. For an ideal MHD case, the selected parameter choices are $Re = Re_m = \infty$, $s=1$ and $f=\nabla \times g = 0$. Consider the following initial conditions
 \begin{align}\label{truesol}
\bu_0=\left(%
\begin{array}{c}
-\sin (y+2) \\
\sin (x+1.4)%
\end{array}%
\right),\quad
B_0=\left(%
\begin{array}{c}
-\frac{1}{3}\sin (y + 6.2) \\
\frac{2}{3}\sin (2x + 2.3)%
\end{array}\right)
\end{align}
along with the periodic boundary conditions. Since an ideal MHD case is assumed, the global energy and the cross helicity defined by
 \begin{eqnarray*}
	E&=&\frac{1}{2}\int_{\Omega}(\bu (x)\bu (x)+ sB(x)B(x))dx,\\
    H&=&\frac{1}{2}\int_{\Omega}(\bu (x)B(x))dx
\end{eqnarray*}
should be conserved through the solutions obtained by  Algorithm \ref{algo}. As depicted in Figure \ref{fig:energyh}, the quantities of interest are exactly conserved, while the backward Euler scheme which consists of only discarding the filters fails to preserve them. We can deduce that the classical backward Euler method ruins the energy and cross helicity properties and time filters correct this behavior. Thus, the results for conserved quantities are consistent with the theory.

It is worth noting that the simulations are ran using the coarse mesh which already provides very similar results as the finest grid $4096\times 4096$ of \cite{friedel} and $1024 \times 1024$ of \cite{LW0}.

\begin{figure}[H]
	\centering
	\includegraphics[width=0.4\linewidth, height=0.25\textheight]{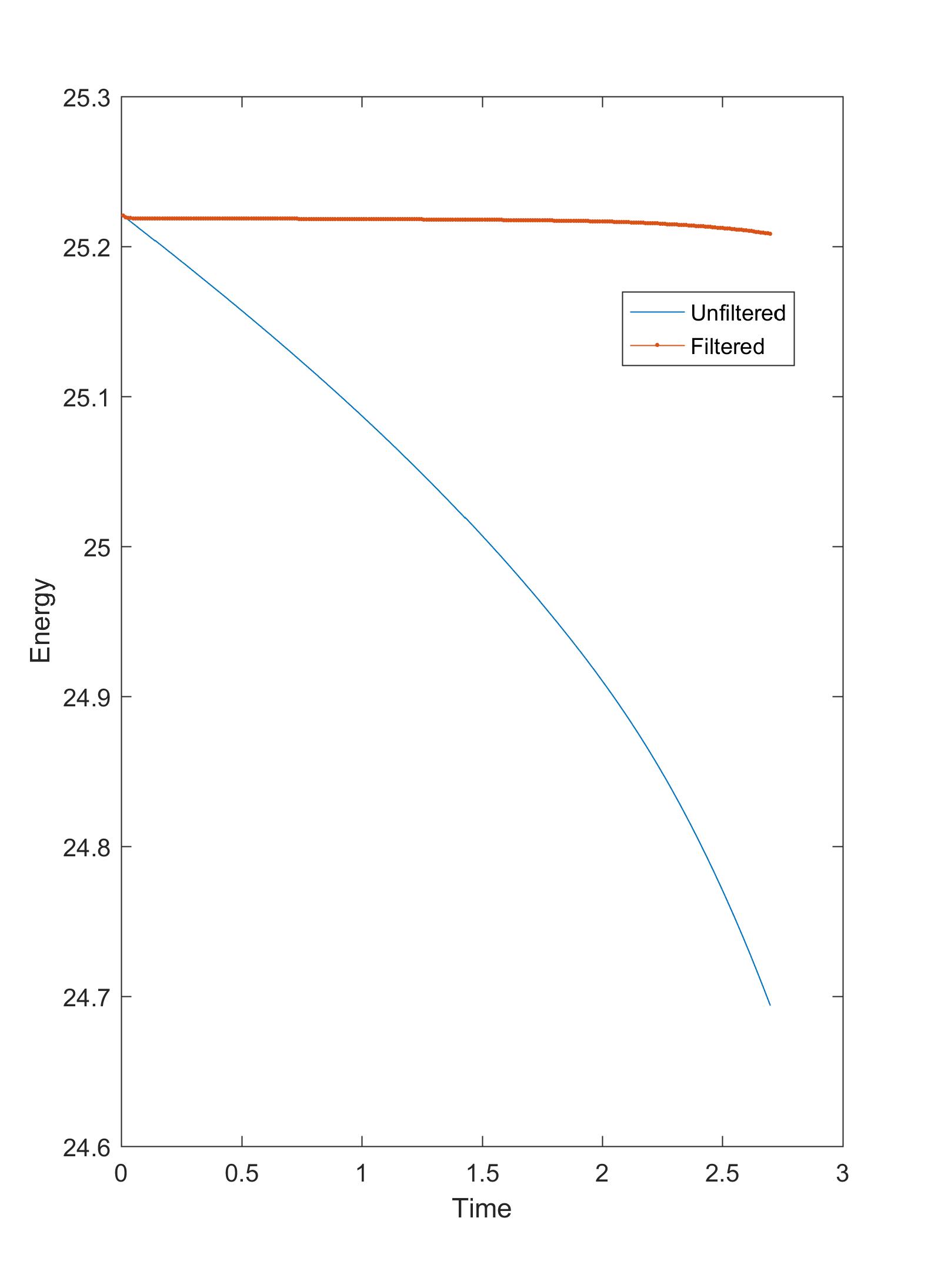}
	\hspace{0.5cm}
		\includegraphics[width=0.4\linewidth, height=0.25\textheight]{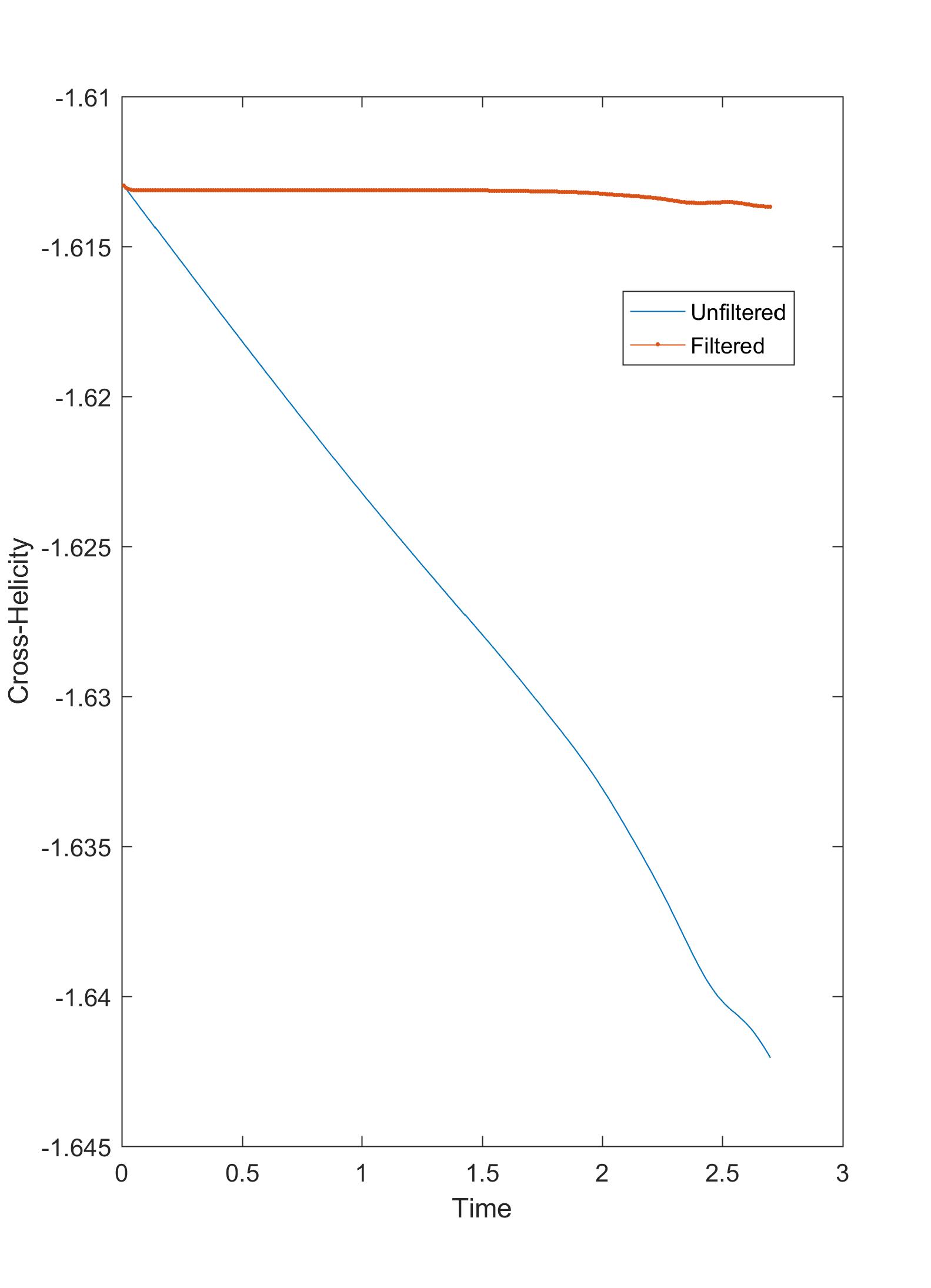}
	\caption{Energy and cross helicity versus time for
backward Euler (unfiltered case) and filtered backward Euler.}
	\label{fig:energyh}
\end{figure}

\subsection{MHD Channel Flow Over a Step}

Our final numerical example is to test Algorithm \ref{algo} for the benchmark  MHD channel flow over past a step. The problem geometry consists of a rectangular $40\times 10$ channel with a $1\times 1$ step places $5$ units into the channel at the bottom. We pick $Re=1000$ and $Re_m=1$ along with varying $s$ and Dirichlet boundary conditions corresponding to no slip velocity on the walls. We impose $\bu=\left( y(10-y)/25, 0 \right)^T$ for the velocity on the inlet and outlet and $\bu=0$ for the rest. For the magnetic field boundary condition, we take $B=\left(0, 1\right)^T$ on all boundaries. As initial conditions, we take $\bu=\left(y(10-y)/25, 0 \right)^T$ and $B=0$. The computations are carried out with $\Delta t =0.025$ up to an end time of $40$ that provides $328,148$ total degrees of freedom. The development of the flow is depicted in {Figure \ref{fig:stream}}. Our interest is only flow behaviour behind the step, thus we present the figures up to $30\times 10$ part of the channel. Note that since there is no magnetic force in the case of $s=0$, we only give velocity streamlines over speed contours.  For $s=0.01$, two eddies start to develop behind the step and the eddies separate from the step between $t=5$ and $t=10$. Due to the effect of the Lorentz force,  the peeling of the eddies behind the step is suppressed for $s=0.05$. As result, the solution captures the correct eddy formation and detachment behind the step. We note that the initial parabolic profile of the initial velocity is changed and the results shown in Figure \ref{fig:stream} are compatible with \cite{mine}.
\begin{figure}[H]
	\centering
		$s=0$\\
	\includegraphics[width=0.9\linewidth, height=0.15\textheight]{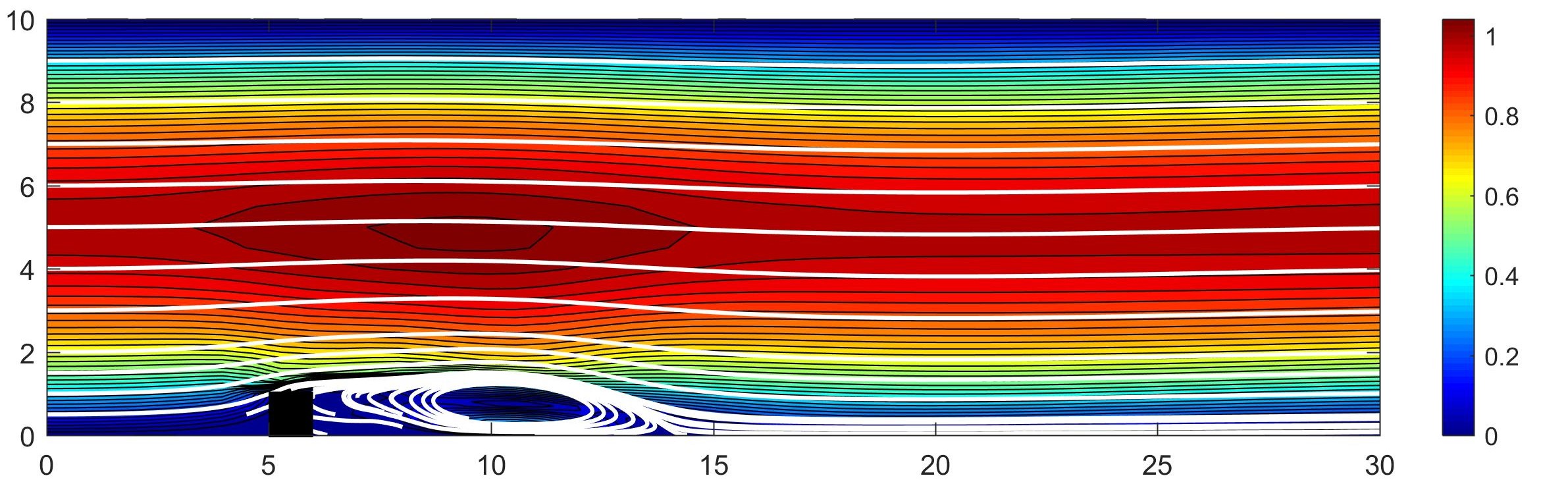}\\
	$s=0.01$\\
	\includegraphics[width=0.9\linewidth, height=0.15\textheight]{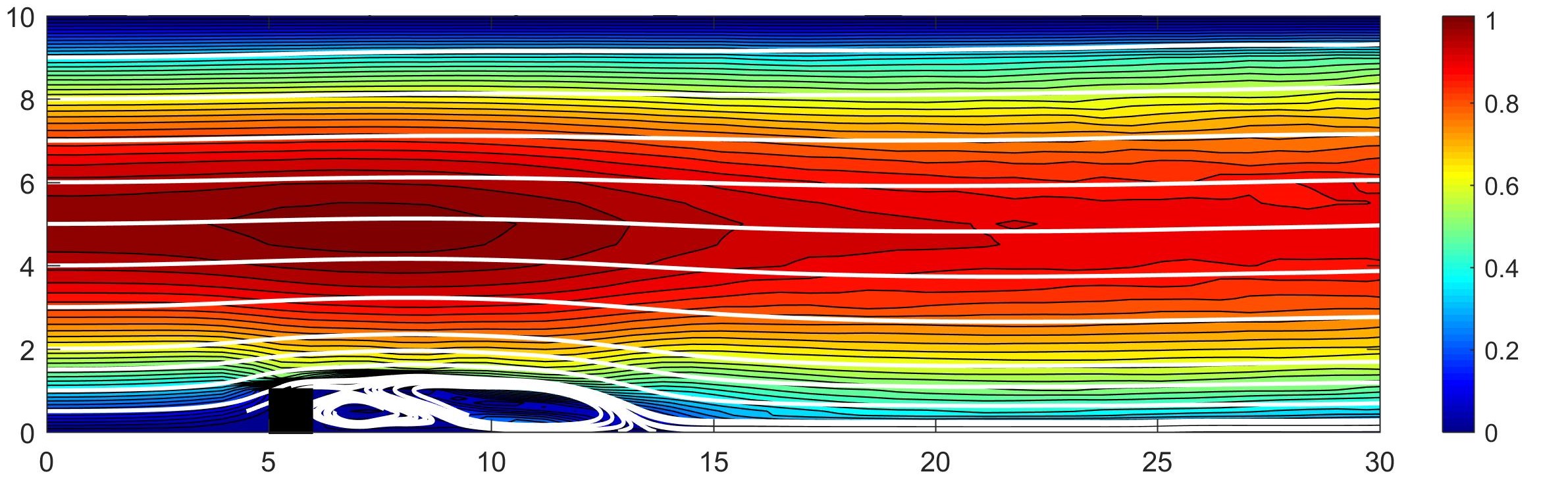}
	\includegraphics[width=0.9\linewidth, height=0.15\textheight]{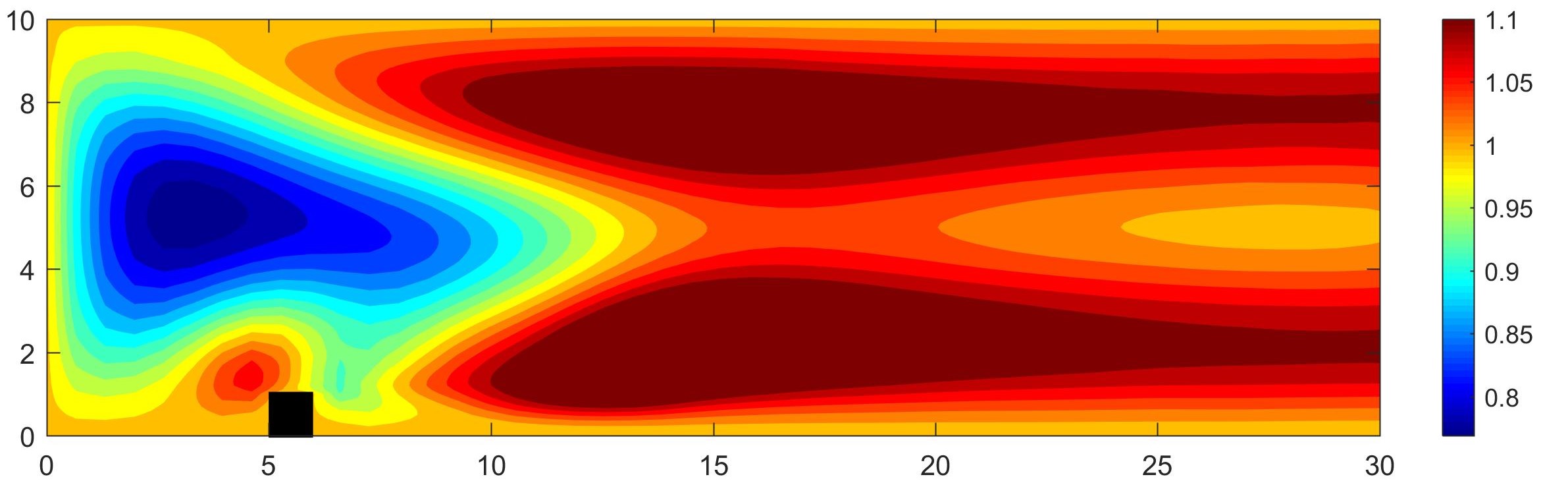}\\
		$s=0.05$\\
		\includegraphics[width=0.9\linewidth, height=0.15\textheight]{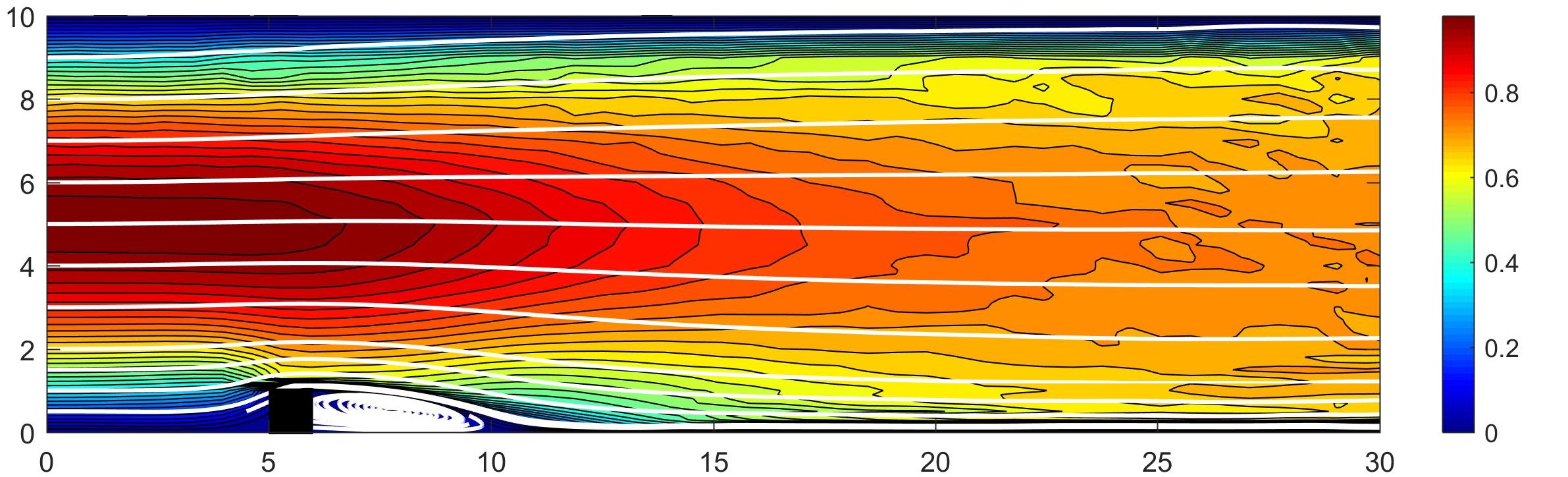}
	\includegraphics[width=0.9\linewidth, height=0.15\textheight]{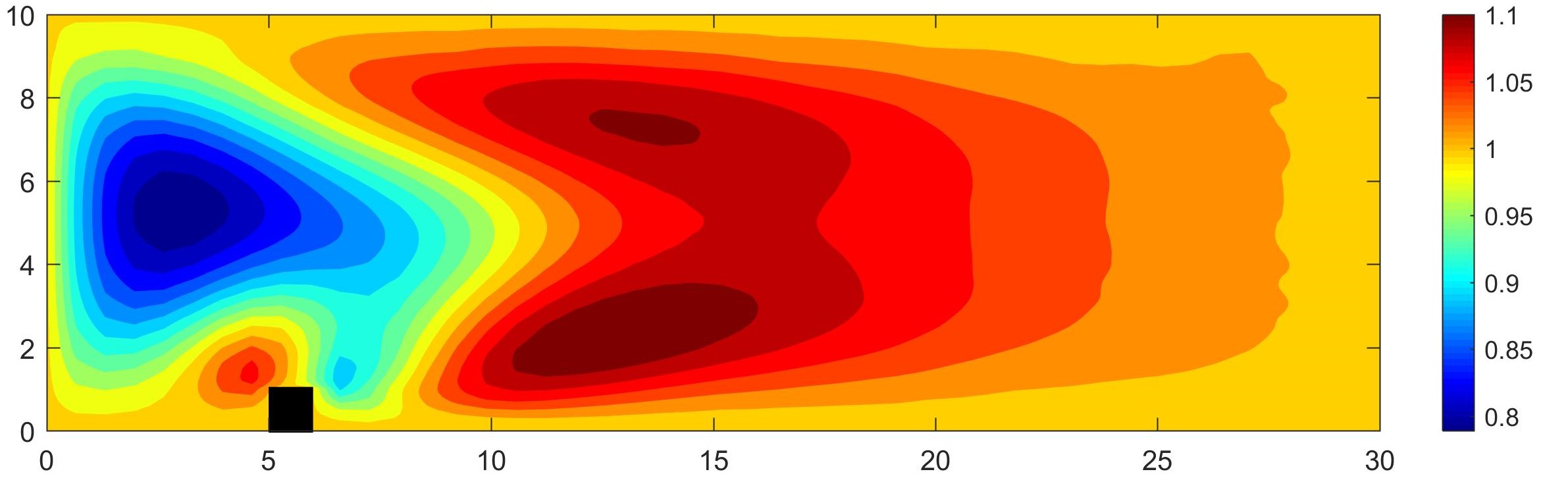}
	\caption{Plots of streamlines over speed and magnetic field contours for varying $s$}
	\label{fig:stream}
\end{figure}
\section{Conclusions}

An efficient time filtered method as a post processing step is introduced to MHD equations in a given backward Euler code. We have shown that the time filtered algorithm increases accuracy from first order to second order without any extra programming effort. We have provided a complete numerical analysis of the method, including unconditional, long time stability and optimal convergence rates. Moreover, time filtered backward Euler method conserves  energy and cross helicity when the solenoidal constraints on the velocity and magnetic field enforced strongly. Results of several numerical tests have been presented in order to verify all theoretical findings. The numerical investigations have shown the time filtered backward Euler method to be very effective and to predict the energy and helicity very well in comparison with the backward Euler method.

Several research directions will be pursued in future. For instance, we will study variable time step methods for MHD, which require only one BDF solve at each time level followed by addition of the solution at previous time steps to develop embedded family of higher order accuracy.  In addition, the extension of time filtering to more complex coupled flow problems such as MHD convection and double diffusive convection will be topics of future research.


\bibliographystyle{amsplain}
\bibliography{reference}
\end{document}